\documentclass{amsart}
\newtheorem{theorem}{Theorem}[section]

\usepackage{latexsym}
\usepackage{amssymb,amsmath,amsfonts}
\usepackage[colorlinks=true]{hyperref}
\usepackage{graphicx}
\usepackage{array}

\newtheorem{ass}[theorem]{Assumption}
\newtheorem{alg}[theorem]{Algorithm}

\newtheorem{lemma}[theorem]{Lemma}

\theoremstyle{definition}

\newtheorem{example}[theorem]{Example}

\theoremstyle{remark}
\newtheorem{remark}[theorem]{Remark}

\numberwithin{equation}{section}
 \global\parskip 6pt \oddsidemargin
0.1cm \evensidemargin 0.4cm \headheight 0.1cm \topmargin -2.0cm
\begin{document}
\title
[Strongly convergent forward-reflected-backward splitting methods]
 {Strong Convergence of Forward-Reflected-Backward Splitting Methods for Solving Monotone Inclusions with Applications to Image Restoration and Optimal Control}

\author{Chinedu Izuchukwu$^1$, Simeon Reich$^{2}$, Yekini Shehu$^3$,  Adeolu Taiwo$^{4}$}

\keywords{Forward-reflected-backward method; inertial method; Halpern's iteration; viscosity iteration; monotone inclusion; strong convergence. \\
{\rm 2010} {\it Mathematics Subject Classification}: 47H09; 47H10; 49J20; 49J40\\\\
{$^{1, 2,4}$Department of Mathematics, The Technion -- Israel Institute of Technology, 32000 Haifa, Israel.}\\
{$^{3}$Department of Mathematics, Zhejiang Normal University, Jinhua 321004, People’s Republic of China.}\\
{$^{1}\mbox{izuchukwu}\_c$@yahoo.com; chi.izuchukw@campus.technion.ac.il}\\
{$^2$sreich@technion.ac.il}\\
{$^{3}$yekini.shehu@zjnu.edu.cn}\\
{$^{4}$ taiwoa@campus.technion.ac.il; taiwo.adeolu@yahoo.com}}
\begin{abstract}
In this paper, we propose and study several strongly convergent versions of the forward-reflected-backward splitting method of Malitsky and Tam for finding a zero of the sum of two monotone operators in a real Hilbert space. Our proposed methods only require one forward evaluation of the single-valued operator and one backward evaluation of the set-valued operator at each iteration; a feature that is absent in many other available strongly convergent splitting methods in the literature. We also develop inertial versions of our methods and strong convergence results are obtained for these methods when the set-valued operator is maximal monotone and the single-valued operator is Lipschitz continuous and monotone. Finally, we discuss some examples from image restorations and optimal control regarding the implementations of our methods in comparisons with known related methods in the literature.
\end{abstract}

\maketitle

\section{Introduction}\label{Se1}
\noindent Let $\mathcal{H}$ be a real Hilbert space. We are interested in the monotone inclusion problem
\begin{eqnarray}\label{MIP}
\mbox{find} ~\widehat{w}\in \mathcal{H}~\mbox{such that}~0\in (S+T)\widehat{w},
\end{eqnarray}
where $S:\mathcal{H}\to2^{\mathcal{H}}$ and $T:\mathcal{H}\to {\mathcal{H}}$ are monotone operators. We denote the solution set of \eqref{MIP} by $(S+T)^{-1}(0)$. Problem \eqref{MIP} naturally includes many important optimization problems such as variational inequalities, minimization problems, linear inverse problems, saddle-point problems, fixed point problems, split feasibility problems, Nash equilibrium problems in noncooperative games, and many more (see \cite{IRS4}). Also, many problems in signal processing, image recovery and machine learning can be formulated as Problem \eqref{MIP}.

\noindent Some of the most commonly used methods for solving the monotone inclusion problem \eqref{MIP} are various splitting methods. These methods involve tackling each of the two operators ($S$ and $T$) separately (by means of forward evaluation of the single-valued operator and backward evaluation of the set-valued operator) rather than their sum. A popular splitting method for solving Problem \eqref{MIP} is the forward-backward splitting method \cite{28Y, 35Y}
\begin{eqnarray}\label{FB}
w_{n+1}=(I_{\mathcal{H}}+\bar{\delta} S)^{-1}(w_n-\bar{\delta}  Tw_n),~n\geq 1,
\end{eqnarray}
where $I_{\mathcal{H}}$ is the identity operator on $\mathcal{H}$ and $\bar{\delta} >0$ is a constant. Method \eqref{FB} requires at each iteration only one forward evaluation of $T$ and one backward evaluation of $S$. This method is known to converge weakly to a solution of Problem \eqref{MIP} when $T$ is $\mathsf{L}^{-1}$-cocoercive,  $\bar{\delta}\in (0, 2\mathsf{L}^{-1})$, $S$ is maximal monotone and $(S+T)^{-1}(0)$ is nonempty. Apart from the cocoercivity assumption on $T$ (which is strict), other assumptions which guarantee the convergence of \eqref{FB} are the strong monotonicity of $S+T$ \cite{12M} or the use of a backtracking technique \cite{6M} (which are also strict).

\noindent In order to weaken the cocoercivity assumption on $T$, the following forward-backward-forward splitting method was introduced in \cite{Tseng}:
\begin{eqnarray}\label{FBF}
\begin{cases}
v_n=(I_{\mathcal{H}}+\bar{\delta} S)^{-1}(w_n-\bar{\delta} Tw_n),\\
w_{n+1}=v_n-\bar{\delta} Tv_n+\bar{\delta} Tw_n,~n\geq 1.
\end{cases}
\end{eqnarray}
This method converges weakly to a solution of \eqref{MIP} under the assumptions that $T$ is monotone and $\mathsf{L}$-Lipschitz continuous, $\bar{\delta}\in (0, \mathsf{L}^{-1})$, $S$ is maximal monotone and $(S+T)^{-1}(0)\neq \emptyset$.
Note that monotonicity and Lipschitz continuity are much weaker assumptions than cocoercivity and strong monotonicity \cite[Remark 2.1]{IRS1}. However, the forward-backward-forward splitting method \eqref{FBF} requires an additional forward evaluation of $T$ (that is, it involves two forward evaluations of $T$ and one backward evaluation of $S$ per iteration). This might affect the efficiency of the method especially when it is applied to solving large-scale optimization problems and optimization problems emanating from optimal control theory, where computations involving pertinent operators are often very expensive (see \cite{19Hv}).\\

\noindent
 Strongly convergent variants of both the forward-backward splitting method \eqref{FB} and the forward-backward-forward splitting method \eqref{FBF} have been studied in the literature recently in \cite{FBFS4, FBFS1, FBFS3, FBFS2, FBS2}. In \cite{FBFS4}, four modifications of the inertial forward-backward
splitting method for solving monotone inclusion problem \eqref{MIP} are discussed. For instance, the authors studied the following inertial viscosity-type forward-backward-forward splitting method  \cite[Algorithm 3.4]{FBFS4}:
\begin{eqnarray}\label{SFBF}
\begin{cases}
u_n=w_n + \vartheta_n(w_n-w_{n-1}),\\
v_n=(I_{\mathcal{H}}+\bar{\delta} S)^{-1}(u_n-\bar{\delta} Tu_n),\\
z_n=v_n-\bar{\delta} Tv_n+\bar{\delta} Tw_n,\\
w_{n+1}=\sigma_n Uw_n+(1-\sigma_n)z_n,~n\geq 1,
\end{cases}
\end{eqnarray}
where $U$ is a contraction mapping  and $\vartheta_n \in [0,1)$ is the inertial parameter. It was shown in \cite[Theorem 3.4]{FBFS4} that $\{w_n\}$ converges strongly to a point in $(S+T)^{-1}(0)$. However, observe that \cite[Algorithm 3.4]{FBFS4} and other strongly convergent splitting methods in \cite{FBFS4, FBFS1, XU, FBFS3, FBS1, FBFS2, FBS2} have the same drawback of at least two forward evaluations of the single-valued operator $T$ per iteration or the assumption that $T$ is cocoercive. \\

\noindent In order to overcome this disadvantage inherent in the forward-backward-forward splitting method \eqref{FBF}, Malitsky and Tam \cite{MT} proposed the following forward-reflected-backward splitting method:
\begin{eqnarray}\label{FRB}
w_{n+1}=(I_{\mathcal{H}}+\bar{\delta} S)^{-1}(w_n-2\bar{\delta} Tw_n+\bar{\delta} Tw_{n-1}),~n\geq 1,
\end{eqnarray}
where $\bar{\delta}\in (0, \frac{1}{2}\mathsf{L}^{-1})$. Method \eqref{FRB} converges weakly to a solution of \eqref{MIP} under the same assumptions as the forward-backward-forward splitting method \eqref{FBF}, but has the same computational structure as the forward-backward splitting method \eqref{FB}. In other words, it converges weakly when $T$ is monotone and Lipschitz continuous, and $S$ is maximal monotone, but only requires one forward evaluation of $T$ and one backward evaluation of $S$. See \cite{MT3, MT1, GRAAL, MT2} and references therein for several weakly convergent variants of Method \eqref{FRB}.\\

\noindent
It is worth noting that in infinite dimensional spaces, strong convergence results are much more desirable than weak convergence results for iterative algorithms. However, due to the computational structure of Method \eqref{FRB}, its strongly convergent variants are very rare in the literature (see, e.g., \cite{HV}). Note that the strongly convergent method proposed in \cite{HV} is of a hybrid projection type and it tackles the particular case where $S$ in the inclusion problem \eqref{MIP} is the normal cone of a nonempty, closed and convex set.
 
\noindent
\textbf{Our Contributions.} Motivated by the above-mentioned results, our contributions in this paper are summarized below.
\begin{itemize}
  \item We introduce new strongly convergent splitting methods for solving the monotone inclusion problem \eqref{MIP}. In similar computational structure as the forward-reflected-backward splitting method \eqref{FRB}, our proposed methods only require one forward evaluation of $T$ and one backward evaluation of $S$ at each iteration, and strong convergence results are obtained without assuming strong monotonicity of either $S$ or $T$. Thus, our results are strongly convergent versions of the weak convergence results obtained in \cite{MT3, MT1, MT}.
  \item Our first method employs the anchored (or Halpern type) extrapolation technique for a general maximal monotone operator $S$ and our convergence analysis is entirely different from that given in \cite{HV}. Moreover, it is known that besides ensuring strong convergence, the anchored extrapolation step improves the convergence rate of iterative methods (see \cite{XU, 70} for details). Viscosity-type approximations are also discussed and strong convergence results are given under mild assumptions.
      \item To further increase the convergence speed of our methods, we propose inertial variants of both the anchored-type and the viscosity-type approximations, and establish their corresponding strong convergent results.
\item We conduct numerical experiments arising from image restoration and optimal control problems in order to illustrate the validity of our proposed methods. Results from the numerical tests show that our methods are efficient, computationally inexpensive and  outperform related methods in the literature.
\end{itemize}

\noindent
\textbf{Organization.} The rest of our paper is organized as follows: Section \ref{Se2} contains basic definitions and results. In Section \ref{Se3} we present the first method of this paper and establish its strong convergence. In Section \ref{Se4} we present an inertial variant of the first method and obtain some convergence results for it. In Section \ref{VIS} we study the viscosity and inertial viscosity variants of our proposed methods.
In Section \ref{Se5} we present several numerical examples which illustrate the implementation of our algorithms in comparison with known methods in the literature. We then make some concluding remarks and suggest possible directions for future research in Section \ref{Se6}.

\hfill

\section{Preliminaries}\label{Se2}
\noindent Throughout this work, $\langle\cdot , \cdot\rangle$ and $\|\cdot \|$ are the inner product and norm, respectively, of a real Hilbert space $\mathcal{H}$.

\noindent The operator $T: \mathcal{H}\rightarrow \mathcal{H}$ is called {$\mathsf{L}$-cocoercive (or inverse strongly monotone)} if there exists  $\mathsf{L}>0$ such that
	\begin{align*}
	\langle T{a}-T{b},{a}-{b}\rangle\geq \mathsf{L}\|T{a}-T{b}\|^2~ \forall a,b \in \mathcal{H},
	\end{align*}
and {monotone} if
	\begin{align*}\langle Ta-Tb,a-b\rangle\geq 0 \; ~ \forall a,b \in \mathcal{H}.
	\end{align*}
The operator $T$ is called $\mathsf{L}$-{Lipschitz continuous} if there exists $\mathsf{L}>0$ such that
	\begin{align*}
	\|Ta-Tb\|\leq \mathsf{L}\|a-b\|~ \forall a,b\in \mathcal{H}.
	\end{align*}
If $\mathsf{L}\in [0, 1)$ then $T$ is a contraction.	
	
\noindent Let $S$ be a set-valued operator $S:\mathcal{H}\rightarrow 2^{\mathcal{H}}$, then $S$ is said to be {monotone} if $$\langle \widehat{a}-\widehat{b}, a-b\rangle\geq 0 \; ~\forall a,b \in \mathcal{H},~\widehat{a}\in Sa,~\widehat{b}\in Sb.$$ The monotone operator $S$ is called {maximal} if the graph $\mathcal{G}(S)$ of $S$, defined by
$$\mathcal{G}(S):=\{(a, \widehat{a})\in \mathcal{H}\times \mathcal{H} : \widehat{a}\in Sa\},$$ is not properly contained in the graph of any other monotone operator.
In other words, $S$ is called a maximal monotone operator if for $(a,\widehat{a})\in \mathcal{H}\times \mathcal{H}$, we have that  $\langle \widehat{a}-\widehat{b}, a-b\rangle \geq 0$ for all $(b,\widehat{b})\in \mathcal{G}(S)$ implies $\widehat{a}\in Sa$.\\
For a set-valued operator $S$, the resolvent associated with it is the mapping $J_{\bar{\delta}}^{S}: \mathcal{H}\rightarrow 2^{\mathcal{H}}$ defined by
\begin{eqnarray*}
J_{\bar{\delta}}^{S}(a) := (I_{\mathcal{H}}+\bar{\delta} S)^{-1}(a), ~~a\in \mathcal{H},~\bar{\delta}>0.
 \end{eqnarray*}
If $S$ is maximal monotone and $T$ is single-valued, then both $J_{\bar{\delta}}^{S}$ and $J_{\bar{\delta}}^{S}(I_{\mathcal{H}}-\bar{\delta} T)$ are single-valued and everywhere defined on $\mathcal{H}$ \cite{IRS1}.

\hfill

\begin{lemma}\label{NL}  The following equalities are true:
\begin{itemize}
 \item[(i)] $2\langle \widehat{a}, \widehat{b}\rangle=\|\widehat{a}\|^{2}+\|\widehat{b}\|^{2}-\|\widehat{a}-\widehat{b}\|^{2}=\|\widehat{a}+\widehat{b}\|^{2}-\|\widehat{a}\|^{2}-\|\widehat{b}\|^{2}~~~\forall~ \widehat{a},\widehat{b}\in \mathcal{H}.$
 \item[(ii)] $\|\bar{r}~\widehat{a}+(1-\bar{r})\widehat{b}\|^2=\bar{r}\|\widehat{a}\|^2+(1-\bar{r})\|\widehat{b}\|^2-\bar{r}(1-\bar{r})\|\widehat{a}-\widehat{b}\|^2~~~\forall~ \widehat{a},\widehat{b}\in \mathcal{H},~\bar{r}\in \mathbb{R}$.
\end{itemize}
\end{lemma}

\begin{lemma}\cite{P49}\label{lem6}
Suppose that $\{t_{n}\}$ is a sequence of nonnegative real numbers, $\{\sigma_{n}\}$ is a sequence of real numbers in $(0, 1)$ satisfying $\sum_{n=1}^\infty \sigma_n=\infty$, and  $\{h_n\}$ is a sequence of real numbers such that $$ t_{n+1}\leq(1-\sigma_{n})t_{n}+\sigma_{n}h_{n},~n\geq 1.$$ If $\limsup\limits_{i\to\infty} h_{n_i}\leq 0$ for each subsequence $\{t_{n_i}\}$ of $\{t_{n}\}$ satisfying $\liminf\limits_{i\to\infty}\left(t_{n_i+1}-t_{n_i}\right)\geq 0$, then $\lim\limits_{n\to\infty} t_n=0$.
\end{lemma}

\begin{lemma}\label{le5}\cite{SY15}
Suppose that $T:\mathcal{H}\rightarrow \mathcal{H}$ is monotone and Lipschitz continuous, and $S:\mathcal{H}\rightarrow2^{\mathcal{H}}$ is maximal monotone, then  $(S+T):\mathcal{H}\rightarrow 2^{\mathcal{H}}$ is maximal monotone.
\end{lemma}

\begin{lemma}\label{l5} \cite[Lem.\ 3.1]{pem}
Suppose that $\{t_{n}\}$ and $\{r_{n}\}$ are sequences of nonnegative real numbers such that
\begin{equation*}
t_{n+1}\leq(1-\sigma_{n})t_{n}+s_{n}+r_{n},\ \ n\geq1,
\end{equation*}
where $\{\sigma_{n}\}$ is a sequence in $(0,1)$ and $\{s_{n}\}$ is a
real sequence. Let $\sum_{n=1}^{\infty}r_{n}<\infty$ and $s_{n}\leq\sigma_{n}M$ for some $M\geq0$. Then $\{t_{n}\}$ is bounded.
 \end{lemma}

\hfill

\section{Modified Forward-Reflected-Backward Splitting Method}\label{Se3}
\hrule\hrule	
	\begin{ass}\label{AS1} ${}$
\hrule\hrule
\begin{itemize}
\item[(a)] $S$ is maximal monotone,
\item[(b)]  $T$ is monotone and Lipschitz continuous with constant $\mathsf{L}>0$,
\item[(c)] $(S+T)^{-1}(0)$ is nonempty.
		\end{itemize}
\end{ass}		

\hfill

\hrule \hrule
\begin{alg}\label{AL1} Let $\delta_0,\delta_1>0$, $\bar{r}\in \Big(\bar{\beta},~\frac{1-2\bar{\beta}}{2}\Big)$ with $\bar{\beta}\in (0, \frac{1}{4})$, and choose the sequences $\{\sigma_n\}$ in $(0, 1)$ and $\{c_n\}$ in $[0, \infty)$ such that $\sum_{n=1}^{\infty} c_n< \infty$. For arbitrary $\hat{v}, w_0,w_1 \in \mathcal{H}$, let the sequence $\{w_n\}$ be generated by
	\hrule \hrule
	 \begin{eqnarray}\label{a1}
w_{n+1}=J^{S}_{\delta_n}\big(\sigma_n \hat{v}+(1-\sigma_n) w_n-\delta_nTw_n-\delta_{n-1}(1-\sigma_n)(Tw_{n}-Tw_{n-1})\big),~ n\geq 1,
\end{eqnarray}
where
\begin{eqnarray}\label{a11}
	\delta_{n+1}=\begin{cases}
	\min \left\{\frac{\bar{r}\|w_n-w_{n+1}\|}{\|Tw_n-Tw_{n+1}\|},~\delta_{n}+c_{n}\right\}, & \mbox{if}~  Tw_n\neq Tw_{n+1},\\
	\delta_{n}+c_n,& \mbox{otherwise}.
	\end{cases}
	\end{eqnarray}
\hrule\hrule
\end{alg}

\noindent We call Algorithm \ref{AL1} a {\it forward-reflected-anchored-backward splitting method} with a self-adaptive step size $\delta_n$, an anchor $\hat{v}$ and an anchoring coefficient $\sigma_n$. Since this algorithm is based on the Halpern  iteration, it can also be viewed as a Halpern-type forward-reflected-backward method. For more information on the convergence of Halpern-type methods for solving optimization problems, see, for example, \cite{24, XU, QY, 70}.

\hfill

\begin{remark} \label{VSS} ${}$
By \eqref{a11}, $\lim\limits_{n\to\infty}\delta_n=\bar{\delta}$, where $\bar{\delta}\in [\min\{\bar{r}\mathsf{L}^{-1}, \delta_1\}, ~\delta_1+\bar{c}]$ with $\bar{c}=\sum_{n=1}^\infty c_n$  (see \cite{HY}).
If $c_n=0$, then the step size $\delta_n$ in \eqref{a11} is similar to the one in \cite{MT1}.
\end{remark}

\hfill

\noindent We now establish the strong convergence of Algorithm \ref{AL1}. We begin with the following lemma.
\begin{lemma}\label{lem1} Let $\{w_n\}$ be generated by Algorithm \ref{AL1} and assume that Assumption \ref{AS1} holds. If $\lim\limits_{n\to\infty}\sigma_n=0$, then $\{w_n\}$ is bounded.
\end{lemma}
	\begin{proof}
		Let $\widehat{w}\in (S+T)^{-1}(0).$  Then $-\delta_n T\widehat{w}\in \delta_nS\widehat{w}.$ Set $a_n:=\sigma_n\hat{v}+(1-\sigma_n) w_n$. Then by \eqref{a1}, we have
		\begin{eqnarray}\label{**}
		a_n-\delta_nTw_n-\delta_{n-1}(1-\sigma_n)(Tw_{n}-Tw_{n-1})-w_{n+1}\in \delta_n  S w_{n+1}.
		\end{eqnarray}
		Thus, by the monotonicity of $S$, we see that
		$$\left\langle a_n-\delta_nTw_n-\delta_{n-1}(1-\sigma_n)(Tw_{n}-Tw_{n-1})-w_{n+1}+\delta_n T\widehat{w},~ w_{n+1}-\widehat{w}\right\rangle \geq 0,$$
	which  implies that
	\begin{eqnarray}\label{M1}
	0&\leq & 2\left\langle w_{n+1}- a_n+\delta_nTw_n+\delta_{n-1}(1-\sigma_n)(Tw_{n}-Tw_{n-1})-\delta_n T\widehat{w},~ \widehat{w}-w_{n+1}\right\rangle\nonumber\\
	&=&2\langle w_{n+1}-a_n,\widehat{w}-w_{n+1}\rangle +2\delta_{n}\langle Tw_n-T\widehat{w},\widehat{w}-w_{n+1}\rangle +2\delta_{n-1}(1-\sigma_n)\langle Tw_n-Tw_{n-1},\widehat{w}-w_{n}\rangle \nonumber\\
&&+2\delta_{n-1}(1-\sigma_n)\langle Tw_n-Tw_{n-1},w_{n}-w_{n+1}\rangle\nonumber\\
&=&\|a_n-\widehat{w}\|^2-\|w_{n+1}-\widehat{w}\|^2-\|w_{n+1}-a_n\|^2 +2\delta_{n}\langle Tw_n-T\widehat{w},\widehat{w}-w_{n+1}\rangle  \nonumber\\
&&+2\delta_{n-1}(1-\sigma_n)\langle Tw_n-Tw_{n-1},\widehat{w}-w_{n}\rangle+2\delta_{n-1}(1-\sigma_n)\langle Tw_n-Tw_{n-1},w_{n}-w_{n+1}\rangle,	
\end{eqnarray}		
where the last equation follows from Lemma \ref{NL}(i). Next, since $T$ is monotone, we have that
\begin{eqnarray}\label{M2}
\langle Tw_n-T\widehat{w}, \widehat{w}-w_{n+1}\rangle\leq \langle Tw_n-Tw_{n+1}, \widehat{w}-w_{n+1}\rangle.
\end{eqnarray}
Also, from \eqref{a11}, we get
	\begin{align}\label{a3A}\nonumber
	2\delta_{n-1}\langle Tw_n-Tw_{n-1},w_n-w_{n+1}\rangle&\leq 2\delta_{n-1} \|Tw_n-Tw_{n-1}\|\|w_n-w_{n+1}\|\\
     &\leq \frac{2\delta_{n-1}}{\delta_n}{\bar{r}}\|w_n-w_{n-1}\|\|w_n-w_{n+1}\|\nonumber\\
	&\leq \frac{\delta_{n-1}}{\delta_n}{\bar{r}}\Big(\|w_n-w_{n-1}\|^2+\|w_{n+1}-w_n\|^2\Big).
	\end{align}
By Remark \ref{VSS} and the condition $\bar{r}\in \Big(\bar{\beta},~\frac{1-2\bar{\beta}}{2}\Big),$ we get $\lim\limits_{n\to\infty}\frac{\delta_{n-1}}{\delta_{n}}\bar{r} =\bar{r}<\frac{1}{2}-\bar{\beta}$. Thus, there exists $n_0\geq 1$ such that $\frac{\delta_{n-1}}{\delta_{n}}\bar{r} <\frac{1}{2}-\bar{\beta}~\forall n\geq n_0$. Hence, using \eqref{M2} and \eqref{a3A} in \eqref{M1},  we get
\begin{eqnarray}\label{LL*}
	&&\|w_{n+1}-\widehat{w}\|^2 +2\delta_{n}\langle Tw_{n+1}-Tw_n,\widehat{w}-w_{n+1}\rangle +\frac{1}{2}\|w_{n+1}-w_n\|^2 \nonumber\\
	&&\leq  \|a_n-\widehat{w}\|^2-\|w_{n+1}-a_n\|^2 +2\delta_{n-1}(1-\sigma_n)\langle Tw_n-Tw_{n-1},\widehat{w}-w_{n}\rangle  \nonumber\\
&&+(1-\sigma_n)\big(\frac{1}{2}-\bar{\beta}\big)\|w_n-w_{n-1}\|^2+\left[\frac{\delta_{n-1}}{\delta_n}{\bar{r}}(1-\sigma_n)+\frac{1}{2}\right]\|w_{n+1}-w_n\|^2~\forall n\geq n_0.	
\end{eqnarray}	
Using Lemma \ref{NL}(i), we see that
\begin{eqnarray}\label{H1}
\|a_n-\widehat{w}\|^2&=&\|(w_n-\widehat{w})-\sigma_n(w_n-\hat{v})\|^2\nonumber\\
&=&\|w_n-\widehat{w}\|^2+\sigma_n^2\|w_n-\hat{v}\|^2-2\sigma_n\langle w_n-\widehat{w}, w_n-\hat{v}\rangle\nonumber\\
&=&\|w_n-\widehat{w}\|^2+\sigma_n^2\|w_n-\hat{v}\|^2-\sigma_n\|w_n-\hat{v}\|^2-\sigma_n\|w_n-\widehat{w}\|^2+\sigma_n\|\hat{v}-\widehat{w}\|^2.
\end{eqnarray}
Replacing $\widehat{w}$ by $w_{n+1}$ in \eqref{H1}, we get
\begin{eqnarray}\label{H2}
\|a_n-w_{n+1}\|^2&=&\|w_n-w_{n+1}\|^2+\sigma_n^2\|w_n-\hat{v}\|^2-\sigma_n\|w_n-\hat{v}\|^2-\sigma_n\|w_n-w_{n+1}\|^2\nonumber \\
&&+\sigma_n\|\hat{v}-w_{n+1}\|^2.
\end{eqnarray}
Now, subtracting \eqref{H2} from \eqref{H1}, we obtain
\begin{eqnarray}\label{H**}
&&\|a_n-\widehat{w}\|^2-\|a_n-w_{n+1}\|^2\nonumber\\
&=&(1-\sigma_n)\|w_n-\widehat{w}\|^2+\sigma_n\|\hat{v}-\widehat{w}\|^2-(1-\sigma_n)\|w_{n+1}-w_n\|^2-\sigma_n\|w_{n+1}-\hat{v}\|^2.
\end{eqnarray}

\noindent Using \eqref{H**} in \eqref{LL*}, we get
\begin{eqnarray}\label{LL**}
	&&\|w_{n+1}-\widehat{w}\|^2 +2\delta_{n}\langle Tw_{n+1}-Tw_n,\widehat{w}-w_{n+1}\rangle +\frac{1}{2}\|w_{n+1}-w_n\|^2 \nonumber\\
	&&\leq  (1-\sigma_n)\|w_n-\widehat{w}\|^2+\sigma_n\|\hat{v}-\widehat{w}\|^2-(1-\sigma_n)\|w_{n+1}-w_n\|^2-\sigma_n\|w_{n+1}-\hat{v}\|^2 \nonumber\\
	&& +2\delta_{n-1}(1-\sigma_n)\langle Tw_n-Tw_{n-1},\widehat{w}-w_{n}\rangle +(1-\sigma_n)\big(\frac{1}{2}-\bar{\beta}\big)\|w_n-w_{n-1}\|^2 \nonumber\\
&&+\left[\frac{\delta_{n-1}}{\delta_n}{\bar{r}}(1-\sigma_n)+\frac{1}{2}\right]\|w_{n+1}-w_n\|^2\nonumber\\
&&\leq  (1-\sigma_n)\left[\|w_n-\widehat{w}\|^2 +2\delta_{n-1}\langle Tw_n-Tw_{n-1},\widehat{w}-w_{n}\rangle +\frac{1}{2}\|w_n-w_{n-1}\|^2\right] \nonumber\\
&&+\sigma_n\|\hat{v}-\widehat{w}\|^2 -\left[\frac{1}{2}-\sigma_n-\frac{\delta_{n-1}}{\delta_n}{\bar{r}}(1-\sigma_n)\right]\|w_{n+1}-w_n\|^2-\bar{\beta}(1-\sigma_n)\|w_n-w_{n-1}\|^2~\forall n\geq n_0.	
\end{eqnarray}	

\noindent Set $t_n:=\|w_n-\widehat{w}\|^2 +2\delta_{n-1}\langle Tw_n-Tw_{n-1},\widehat{w}-w_{n}\rangle +\frac{1}{2}\|w_n-w_{n-1}\|^2$. Then for all $n\geq n_0$,
\begin{eqnarray}\label{Pos}
t_n&\geq & \|w_n-\widehat{w}\|^2 -2\delta_{n-1}\| Tw_n-Tw_{n-1}\| \|w_{n}-\widehat{w}\| +\frac{1}{2}\|w_n-w_{n-1}\|^2\nonumber\\
&\geq & \|w_n-\widehat{w}\|^2 -\frac{\delta_{n-1}}{\delta_{n}} \bar{r}\left(\| w_n-w_{n-1}\|^2+ \|w_{n}-\widehat{w}\|^2\right) +\frac{1}{2}\|w_n-w_{n-1}\|^2\nonumber\\
&\geq & \|w_n-\widehat{w}\|^2 +(\bar{\beta}-\frac{1}{2}) \left(\| w_n-w_{n-1}\|^2+ \|w_{n}-\widehat{w}\|^2\right) +\frac{1}{2}\|w_n-w_{n-1}\|^2\nonumber\\
&=& (\frac{1}{2}+\bar{\beta})\|w_{n}-\widehat{w}\|^2+\bar{\beta}\|w_n-w_{n-1}\|^2.
\end{eqnarray}
Hence, $t_n\geq 0$ for all $n\geq n_0$. On the other hand, since $\lim\limits_{n\to\infty}\sigma_n=0$ and $\bar{r}<\frac{1-2\bar{\beta}}{2}$, we have
$$\lim\limits_{n\to\infty}\left[\frac{1}{2}-\sigma_n-\frac{\delta_{n-1}}{\delta_n}{\bar{r}}(1-\sigma_n)\right]=\frac{1}{2}-\bar{r}>\bar{\beta}.$$
Thus, there exists $n_1\in \mathbb{N}$, $n_1\geq n_0$, such that $\frac{1}{2}-\sigma_n-\frac{\delta_{n-1}}{\delta_n}{\bar{r}}(1-\sigma_n)>\bar{\beta}~\forall n\geq n_1.$ Therefore, it follows from \eqref{LL**} that
\begin{eqnarray}
t_{n+1}&\leq & (1-\sigma_n)t_n +\sigma_n\|\hat{v}-\widehat{w}\|^2-\bar{\beta}\|w_{n+1}-w_n\|^2-\bar{\beta}(1-\sigma_n)\|w_n-w_{n-1}\|^2\nonumber\\
&\leq & (1-\sigma_n)t_n +\sigma_n\|\hat{v}-\widehat{w}\|^2 ~\forall n\geq n_1\geq n_0,
\end{eqnarray}
which by Lemma \ref{l5} implies that $\{t_n\}$ is bounded. It follows from \eqref{Pos} that $\{w_n\}$ is bounded too.
\end{proof}

\hfill

\begin{theorem}\label{THM1} Let $\{w_n\}$ be generated by Algorithm \ref{AL1} when Assumption \ref{AS1} holds. If $\lim\limits_{n\to\infty}\sigma_n=0$ and $\sum\limits_{n=1}^\infty\sigma_n=\infty$, then $\{w_n\}$ converges strongly to $P_{(S+T)^{-1}(0)}\hat{v}$.
\end{theorem}
\begin{proof}
Let $\widehat{w}=P_{(S+T)^{-1}(0)}\hat{v}$. Using Lemma \ref{NL}(i), we obtain
\begin{eqnarray}\label{T1}
\|a_n-\widehat{w}\|^2&=&\|\sigma_n(\hat{v}-\widehat{w})+(1-\sigma_n)(w_n-\widehat{w})\|^2\nonumber\\
&=&\sigma_n^2\|\hat{v}-\widehat{w}\|^2+(1-\sigma_n)^2\|w_n-\widehat{w}\|^2+2\sigma_n(1-\sigma_n)\langle \hat{v}-\widehat{w}, w_n-\widehat{w}\rangle
\end{eqnarray}
Similarly, we obtain
\begin{eqnarray}\label{T2}
\|a_n-w_{n+1}\|^2&=&\sigma_n^2\|\hat{v}-w_{n+1}\|^2+(1-\sigma_n)^2\|w_n-w_{n+1}\|^2+2\sigma_n(1-\sigma_n)\langle \hat{v}-w_{n+1}, w_n-w_{n+1}\rangle\nonumber\\
&\geq &\sigma_n^2\|w_{n+1}-\hat{v}\|^2+(1-\sigma_n)^2\|w_{n+1}-w_n\|^2-2\sigma_n(1-\sigma_n)\|w_{n+1}-\hat{v}\| \|w_{n+1}-w_n\|\nonumber\\
&\geq &\sigma_n^2\|w_{n+1}-\hat{v}\|^2+(1-\sigma_n)^2\|w_{n+1}-w_n\|^2-2\sigma_n(1-\sigma_n)M \|w_{n+1}-w_n\|,
\end{eqnarray}
where $M:=\sup\limits_{n\geq 1}\|w_{n+1}-\hat{v}\|$ (recall that in view of Lemma \ref{lem1}, the sequence $\{w_n\}$ is bounded). Now, using \eqref{T1} and \eqref{T2} in \eqref{LL*}, we see that
\begin{eqnarray}\label{ABC}
	&&\|w_{n+1}-\widehat{w}\|^2 +2\delta_{n}\langle Tw_{n+1}-Tw_n,\widehat{w}-w_{n+1}\rangle +\frac{1}{2}\|w_{n+1}-w_n\|^2 \nonumber\\
	&&\leq  \sigma_n^2\|\hat{v}-\widehat{w}\|^2+(1-\sigma_n)^2\|w_n-\widehat{w}\|^2+2\sigma_n(1-\sigma_n)\langle \hat{v}-\widehat{w}, w_n-\widehat{w}\rangle-\sigma_n^2\|w_{n+1}-\hat{v}\|^2\nonumber\\
	&&-(1-\sigma_n)^2\|w_{n+1}-w_n\|^2+2\sigma_n(1-\sigma_n)M \|w_{n+1}-w_n\|+2\delta_{n-1}(1-\sigma_n)\langle Tw_n-Tw_{n-1},\widehat{w}-w_{n}\rangle \nonumber\\
	&&  +(1-\sigma_n)\big(\frac{1}{2}-\bar{\beta}\big)\|w_n-w_{n-1}\|^2+\left[\frac{\delta_{n-1}}{\delta_n}{\bar{r}}(1-\sigma_n)+\frac{1}{2}\right]\|w_{n+1}-w_n\|^2\nonumber\\
	&&\leq  (1-\sigma_n) \left(\|w_n-\widehat{w}\|^2 +2\delta_{n-1}\langle Tw_n-Tw_{n-1},\widehat{w}-w_{n}\rangle +\frac{1}{2}\|w_{n}-w_{n-1}\|^2\right) \nonumber\\
	&&+\sigma_n\left(\sigma_n\|\hat{v}-\widehat{w}\|^2+2(1-\sigma_n)\langle \hat{v}-\widehat{w}, w_n-\widehat{w}\rangle+2(1-\sigma_n)M \|w_{n+1}-w_n\|\right) \nonumber\\
	&& -\left(\frac{1}{2}+\sigma_n^2-2\sigma_n-\frac{\delta_{n-1}}{\delta_n}{\bar{r}}(1-\sigma_n)\right)\|w_{n+1}-w_n\|^2-\bar{\beta}(1-\sigma_n)\|w_n-w_{n-1}\|^2~\forall n\geq n_1.	
\end{eqnarray}	

\noindent Therefore, for all $n\geq n_1$, we have
\begin{eqnarray}\label{STRa}
t_{n+1} &\leq & (1-\sigma_n)t_n+\sigma_n h_n,
\end{eqnarray}
where $h_n=\sigma_n\|\hat{v}-\widehat{w}\|^2+2(1-\sigma_n)\langle \hat{v}-\widehat{w}, w_n-\widehat{w}\rangle+2(1-\sigma_n)M \|w_{n+1}-w_n\|$.

\noindent To conclude the proof, it suffices to show, in view of Lemma \ref{lem6}, that $\limsup\limits_{i\to\infty}h_{n_i}\leq 0$ for each subsequence $\{t_{n_i}\}$ of $\{t_n\}$ such that $\liminf\limits_{i\to \infty}\left(t_{n_i+1}-t_{n_i}\right)\geq 0$.
To this end, let $\{t_{n_i}\}$ be a subsequence of $\{t_n\}$ such that $\liminf\limits_{i\to \infty}\left(t_{n_i+1}-t_{n_i}\right)\geq 0$.
Using \eqref{ABC}, we obtain
\begin{eqnarray*}
&&\limsup_{i\to\infty}\left[\left(\frac{1}{2}+\sigma_{n_i}^2-2\sigma_{n_i}-\frac{\delta_{{n_i}-1}}{\delta_{n_i}}{\bar{r}}(1-\sigma_{n_i})\right)\|w_{{n_i}+1}-w_{n_i}\|^2\right]\\
&&\leq \limsup_{i\to \infty}\left[(t_{n_i}-t_{{n_i}+1})+\sigma_{n_i} (h_{n_i}-t_{n_i})\right]\\
&&\leq -\liminf_{i\to \infty}(t_{{n_i}+1}-t_{n_i})\leq 0.
\end{eqnarray*}
Since $\lim\limits_{i\to\infty}\left(\frac{1}{2}+\sigma_{n_i}^2-2\sigma_{n_i}-\frac{\delta_{{n_i}-1}}{\delta_{n_i}}{\bar{r}}(1-\sigma_{n_i})\right)=\frac{1}{2}-\bar{r}>0$, we get
\begin{eqnarray}\label{Z1}
\lim_{i\to\infty}\|w_{{n_i}+1}-w_{n_i}\|=0.
\end{eqnarray}
Also,
\begin{eqnarray}\label{Z2}
\lim_{i\to\infty} \|a_{n_i}-w_{n_i}\|=\lim_{i\to\infty}\sigma_{n_i}\|w_{n_i}-\hat{v}\|=0.
\end{eqnarray}
Using \eqref{Z1} and \eqref{Z2}, we find that
\begin{eqnarray}\label{Z3}
\lim_{i\to\infty}\|w_{{n_i}+1}-a_{n_i}\|=0.
\end{eqnarray}
Combining the Lipschitz continuity of $T$ and \eqref{Z1}, we get
\begin{eqnarray}\label{Z4}
\lim_{i\to\infty}\|Tw_{{n_i}+1}-Tw_{n_i}\|=0.
\end{eqnarray}
In view of Lemma \ref{lem1}, $\{w_{n_i}\}$ is bounded. Thus, we can choose a subsequence $\{w_{n_{i_j}}\}$ of $\{w_{n_i}\}$ which converges weakly to some $w^*\in \mathcal{H}$ such that
\begin{eqnarray}\label{Z5}
\limsup_{i\to\infty}\langle \hat{v}-\widehat{w}, w_{n_i}-\widehat{w}\rangle=\lim_{j\to\infty}\langle \hat{v}-\widehat{w}, w_{{n_i}_j}-\widehat{w}\rangle=\langle \hat{v}-\widehat{w}, w^*-\widehat{w}\rangle.
\end{eqnarray}

\noindent Now, consider $(x, y)\in \mathcal{G}(S+T)$. Then $\delta_{n_{i_j}}(y-Tx)\in \delta_{n_{i_j}}Sx$. Using this, \eqref{**} and the monotonicity of $S$, we find that
\begin{eqnarray*}
\langle \delta_{n_{i_j}}(y-Tx) -a_{n_{i_j}}+\delta_{n_{i_j}}Tw_{n_{i_j}}+\delta_{{n_{i_j}}-1}(1-\sigma_{n_{i_j}})(Tw_{n_{i_j}}-Tw_{{n_{i_j}}-1})+w_{{n_{i_j}}+1}, x - w_{{n_{i_j}}+1}\rangle \geq 0.
\end{eqnarray*}
Thus, using the monotonicity of $T$, we obtain
\begin{eqnarray}\label{a13}
\langle y, x-w_{{n_{i_j}}+1}\rangle &\geq& \frac{1}{ \delta_{n_{i_j}}} \langle \delta_{n_{i_j}}Tx+a_{n_{i_j}}-\delta_{n_{i_j}}Tw_{n_{i_j}}-\delta_{{n_{i_j}}-1}(1-\sigma_{n_{i_j}})(Tw_{n_{i_j}}-Tw_{{n_{i_j}}-1})-w_{{n_{i_j}}+1}, x - w_{n_j+1}\rangle\nonumber\\
&=& \langle Tx-Tw_{{n_{i_j}}+1}, x - w_{{n_{i_j}}+1}\rangle + \langle Tw_{{n_{i_j}}+1}-Tw_{n_{i_j}}, x - w_{{n_{i_j}}+1}\rangle\nonumber\\
&& +\frac{\delta_{{n_{i_j}}-1}}{ \delta_{n_{i_j}}}(1-\sigma_{n_{i_j}})\langle Tw_{{n_{i_j}}-1}-Tw_{n_{i_j}}, x - w_{{n_{i_j}}+1}\rangle + \frac{1}{ \delta_{n_{i_j}}}\langle a_{n_{i_j}}-w_{{n_{i_j}}+1}, x - w_{{n_{i_j}}+1}\rangle\nonumber\\
&\geq &   \langle Tw_{{n_{i_j}}+1}-Tw_{n_{i_j}}, x - w_{{n_{i_j}}+1}\rangle +\frac{\delta_{{n_{i_j}}-1}}{ \delta_{n_{i_j}}}(1-\sigma_{n_{i_j}})\langle Tw_{{n_{i_j}}-1}-Tw_{n_{i_j}}, x - w_{{n_{i_j}}+1}\rangle\nonumber\\
&&  + \frac{1}{ \delta_{n_{i_j}}}\langle a_{n_{i_j}}-w_{{n_{i_j}}+1}, x - w_{{n_{i_j}}+1}\rangle.
\end{eqnarray}
As $j\to \infty$ in \eqref{a13}, we get, using \eqref{Z3} and \eqref{Z4}, that $\langle y, x-w^*\rangle\geq 0$. Thus, since $S+T$ is maximal monotone  (see Lemma \ref{le5}), we get that $w^*\in (S+T)^{-1}(0)$.\\
Since $\widehat{w}=P_{(S+T)^{-1}(0)}\hat{v}$, it follows from \eqref{Z5} and the characterization of the metric projection that
\begin{eqnarray}\label{Z6a}
\limsup_{i\to\infty}\langle \hat{v}-\widehat{w}, w_{n_i}-\widehat{w}\rangle=\langle \hat{v}-\widehat{w}, w^*-\widehat{w}\rangle\leq 0.
\end{eqnarray}
Using \eqref{Z1}, \eqref{Z6a} and the condition $\lim\limits_{i\to\infty}\sigma_{n_{i}}=0$, we find that $\limsup\limits_{i\to\infty}h_{n_{i}}\leq0$. Thus, in view of the condition $\sum\limits_{n=1}^\infty\sigma_n=\infty$ and Lemma \ref{lem6}, it follows from \eqref{STRa} that  $\lim\limits_{n\to\infty}t_n=0$. Hence, using \eqref{Pos}, we conclude that $\{w_n\}$ converges strongly to $\widehat{w}=P_{(S+T)^{-1}(0)}\hat{v}$, as asserted.
\end{proof}

\hfill

\noindent The step size defined in \eqref{a11} makes it possible for  Algorithm \ref{AL1} to be applied in practice even when the Lipschitz constant $\mathsf{L}$ of $T$ is not known. However, when this constant is known or can be calculated, we simply adopt the following variant of Algorithm \ref{AL1}:

\hfill

\hrule \hrule
\begin{alg}\label{AL1A**}  Let $\bar{\delta}\in \Big(0,~\frac{1}{2L}\Big)$ and choose the sequence $\{\sigma_n\}$ in $(0, 1)$. For arbitrary $\hat{v}, w_0,w_1 \in \mathcal{H}$, let the sequence $\{w_n\}$ be generated by
	\hrule \hrule
	 \begin{eqnarray*}
w_{n+1}=J^{S}_{\bar{\delta}}\big(\sigma_n \hat{v}+(1-\sigma_n) w_n-\bar{\delta} Tw_n-\bar{\delta}(1-\sigma_n)(Tw_{n}-Tw_{n-1})\big),~ n\geq 1.
\end{eqnarray*}
	\hrule\hrule
\end{alg}

\hfill

\begin{remark}
Using arguments similar to those in Lemma \ref{lem1} and Theorem \ref{THM1},  we can establish that the sequence $\{w_n\}$ generated by Algorithm \ref{AL1A**} converges strongly to $P_{(S+T)^{-1}(0)}\hat{v}$.
\end{remark}

\hfill

\section{Inertial Modified Forward-Reflected-Backward Splitting Method}\label{Se4}
\noindent In this section we first propose and then study the following inertial variant of Algorithm \ref{AL1}.
\hrule \hrule
\begin{alg}\label{AL2} Let $\delta_0,\delta_1>0$, $\bar{\vartheta}\in [0, 1)$, $\bar{r}\in \Big(\bar{\beta},~\frac{1-2\bar{\beta}}{2}\Big)$ with $\bar{\beta}\in (0, \frac{1}{4})$, and choose sequences $\{\sigma_n\}$ in $(0, 1)$ and $\{c_n\}$ in $[0, \infty)$ such that $\sum_{n=1}^{\infty} c_n< \infty$. For arbitrary $\hat{v}, w_0,w_1 \in \mathcal{H}$, let the sequence $\{w_n\}$ be generated by
	\hrule \hrule
	 \begin{eqnarray*}
w_{n+1}=J^{S}_{\delta_n}\left(\sigma_n \hat{v}+(1-\sigma_n)(w_n+\bar{\vartheta}(w_n-w_{n-1}))-\delta_nTw_n-\delta_{n-1}(1-\sigma_n)(Tw_{n}-Tw_{n-1})\right),~\forall n\geq 1,
\end{eqnarray*}
where the step size $\delta_n$ is defined by \eqref{a11}.
\hrule\hrule
\end{alg}

\noindent Algorithm \ref{AL2} combines the anchored step, inertial extrapolation step and the forward-reflected-backward splitting technique. Therefore, it can be referred to as an {\it inertial forward-reflected-anchored-backward splitting method}.

\hfill

\begin{lemma}\label{lem2} Let $\{w_n\}$ be generated by Algorithm \ref{AL2} when Assumption \ref{AS1} holds. If $0\leq \bar{\vartheta}<\min\left\lbrace\frac{\bar{\beta}}{2},~\frac{\frac{1}{2}-\bar{r}}{2}\right\rbrace$ and $\lim\limits_{n\to\infty}\sigma_n=0$, then the sequence $\{w_n\}$ is bounded.
\end{lemma}
	\begin{proof}
		Let $\widehat{w}\in (S+T)^{-1}(0)$ and set $a_n:=\sigma_n\hat{v}+(1-\sigma_n) b_n$, where $b_n=w_n+\bar{\vartheta}(w_n-w_{n-1})$. Then by arguments similar to those used in the proof of Lemma \ref{lem1} up to \eqref{LL**}, we obtain
	\begin{eqnarray}\label{LL**2}
	&&\|w_{n+1}-\widehat{w}\|^2 +2\delta_{n}\langle Tw_{n+1}-Tw_n,\widehat{w}-w_{n+1}\rangle +\frac{1}{2}\|w_{n+1}-w_n\|^2 \nonumber\\
	&&\leq  (1-\sigma_n)\left[\|b_n-\widehat{w}\|^2 +2\delta_{n-1}\langle Tw_n-Tw_{n-1},\widehat{w}-w_{n}\rangle +\frac{1}{2}\|w_n-w_{n-1}\|^2\right] \nonumber\\
&&+\sigma_n\|\hat{v}-\widehat{w}\|^2-(1-\sigma_n)\|w_{n+1}-b_n\|^2 +\left[\frac{\delta_{n-1}}{\delta_n}{\bar{r}}(1-\sigma_n)+\frac{1}{2}\right]\|w_{n+1}-w_n\|^2\nonumber\\
&&-\bar{\beta}(1-\sigma_n)\|w_n-w_{n-1}\|^2~\forall n\geq n_0.	
\end{eqnarray}	
It follows from Lemma \ref{NL}(ii) that
\begin{eqnarray}\label{NL1}
\|b_n-\widehat{w}\|^2&=&\|(1+\bar{\vartheta})(w_n-\widehat{w})-\bar{\vartheta}(w_{n-1}-\widehat{w})\|^2\nonumber\\
&=&(1+\bar{\vartheta})\|w_n-\widehat{w}\|^2-\bar{\vartheta}\|w_{n-1}-\widehat{w}\|^2+\bar{\vartheta}(1+\bar{\vartheta})\|w_n-w_{n-1}\|^2.
\end{eqnarray}
From Lemma \ref{NL}(i) it follows that
\begin{eqnarray}\label{NL2}
\|w_{n+1}-b_n\|^2&=&\|(w_{n+1}-w_n)-\bar{\vartheta}(w_n-w_{n-1})\|^2\nonumber\\
&=&\|w_{n+1}-w_n\|^2+\bar{\vartheta}^2\|w_n-w_{n-1}\|^2-2\bar{\vartheta}\langle w_{n+1}-w_n, w_n-w_{n-1}\rangle\nonumber\\
&\geq &(1-\bar{\vartheta})\|w_{n+1}-w_n\|^2+(\bar{\vartheta}^2-\bar{\vartheta})\|w_n-w_{n-1}\|^2.
\end{eqnarray}
Using \eqref{NL1} and \eqref{NL2} in \eqref{LL**2}, we find that
\begin{eqnarray*}
	&&\|w_{n+1}-\widehat{w}\|^2 +2\delta_{n}\langle Tw_{n+1}-Tw_n,\widehat{w}-w_{n+1}\rangle +\frac{1}{2}\|w_{n+1}-w_n\|^2 \nonumber\\
	&&\leq  (1-\sigma_n)\left[(1+\bar{\vartheta})\|w_n-\widehat{w}\|^2 -\bar{\vartheta}\|w_{n-1}-\widehat{w}\|^2+2\delta_{n-1}\langle Tw_n-Tw_{n-1},\widehat{w}-w_{n}\rangle +\frac{1}{2}\|w_n-w_{n-1}\|^2\right] \nonumber\\
&&+(1-\sigma_n)\bar{\vartheta}(1+\bar{\vartheta})\|w_n-w_{n-1}\|^2+\sigma_n\|\hat{v}-\widehat{w}\|^2-(1-\sigma_n)(1-\bar{\vartheta})\|w_{n+1}-w_n\|^2\nonumber\\
&&-(1-\sigma_n)(\bar{\vartheta}^2-\bar{\vartheta})\|w_n-w_{n-1}\|^2+\left[\frac{\delta_{n-1}}{\delta_n}{\bar{r}}(1-\sigma_n)+\frac{1}{2}\right]\|w_{n+1}-w_n\|^2-\bar{\beta}(1-\sigma_n)\|w_n-w_{n-1}\|^2.
\end{eqnarray*}	
This implies that
\begin{eqnarray}\label{LL**2*}
	&&\|w_{n+1}-\widehat{w}\|^2-\bar{\vartheta}\|w_n-\widehat{w}\|^2 +2\delta_{n}\langle Tw_{n+1}-Tw_n,\widehat{w}-w_{n+1}\rangle +\frac{1}{2}\|w_{n+1}-w_n\|^2 \nonumber\\
	&&\leq  (1-\sigma_n)\left[\|w_n-\widehat{w}\|^2-\bar{\vartheta}\|w_{n-1}-\widehat{w}\|^2 +2\delta_{n-1}\langle Tw_n-Tw_{n-1},\widehat{w}-w_{n}\rangle +\frac{1}{2}\|w_n-w_{n-1}\|^2\right] \nonumber\\
&&+\sigma_n\|\hat{v}-\widehat{w}\|^2-(1-\sigma_n)(\bar{\beta}-2\bar{\vartheta})\|w_n-w_{n-1}\|^2\nonumber\\
&&-\left[(1-\sigma_n)(1-\bar{\vartheta})-\frac{\delta_{n-1}}{\delta_n}{\bar{r}}(1-\sigma_n)-\frac{1}{2}\right]\|w_{n+1}-w_n\|^2~\forall n\geq n_0.
\end{eqnarray}

\noindent Set $\hat{t}_n:=\|w_n-\widehat{w}\|^2-\bar{\vartheta}\|w_{n-1}-\widehat{w}\|^2 +2\delta_{n-1}\langle Tw_n-Tw_{n-1},\widehat{w}-w_{n}\rangle +\frac{1}{2}\|w_n-w_{n-1}\|^2$. \\
Then for all $n\geq n_0$, we have
\begin{eqnarray}\label{Pos2}
\hat{t}_n&\geq & \|w_n-\widehat{w}\|^2-\bar{\vartheta}\|w_{n-1}-\widehat{w}\|^2 -\frac{\delta_{n-1}}{\delta_{n}} \bar{r}\left(\| w_n-w_{n-1}\|^2+ \|w_{n}-\widehat{w}\|^2\right) +\frac{1}{2}\|w_n-w_{n-1}\|^2\nonumber\\
&= & \left(1-\frac{\delta_{n-1}}{\delta_{n}} \bar{r}\right)\|w_n-\widehat{w}\|^2-\bar{\vartheta}\|w_{n-1}-\widehat{w}\|^2  +\left(\frac{1}{2}-\frac{\delta_{n-1}}{\delta_{n}} \bar{r}\right)\|w_n-w_{n-1}\|^2.
\end{eqnarray}
It follows from Lemma \ref{NL}(i) that
\begin{eqnarray}\label{NL3}
\|w_{n-1}-\widehat{w}\|^2&=&\|(w_{n-1}-w_n)+(w_n-\widehat{w})\|^2\nonumber\\
&=&\|w_{n-1}-w_n\|^2+\|w_n-\widehat{w}\|^2+2\langle w_{n-1}-w_n, w_n-\widehat{w}\rangle\nonumber\\
&\leq & 2\|w_n-w_{n-1}\|^2+2\|w_n-\widehat{w}\|^2.
\end{eqnarray}
Now, using \eqref{NL3} in \eqref{Pos2}, we see that
\begin{eqnarray}\label{Pos2*}
\hat{t}_n&\geq  & \left(1-\frac{\delta_{n-1}}{\delta_{n}} \bar{r}\right)\|w_n-\widehat{w}\|^2-\bar{\vartheta}\left[2\|w_n-w_{n-1}\|^2+2\|w_n-\widehat{w}\|^2\right] +\left(\frac{1}{2}-\frac{\delta_{n-1}}{\delta_{n}} \bar{r}\right)\|w_n-w_{n-1}\|^2\nonumber\\
&\geq  & \left(\frac{1}{2}-\frac{\delta_{n-1}}{\delta_{n}} \bar{r}-2\bar{\vartheta}\right)\left[\|w_n-\widehat{w}\|^2+\|w_n-w_{n-1}\|^2\right].
\end{eqnarray}
Since $\bar{\vartheta}<\frac{\frac{1}{2}-\bar{r}}{2}$, we have
$\lim\limits_{n\to\infty}\left(\frac{1}{2}-\frac{\delta_{n-1}}{\delta_{n}} \bar{r}-2\bar{\vartheta}\right)=\frac{1}{2}-\bar{r}-2\bar{\vartheta}>0$. Thus, there exists $n_1\in \mathbb{N}$, $n_1\geq n_0$ such that $\frac{1}{2}-\frac{\delta_{n-1}}{\delta_{n}} \bar{r}-2\bar{\vartheta}>0~\forall n\geq n_1.$ Hence, we get from \eqref{Pos2*} that $\hat{t}_n\geq 0$ for all $n\geq n_1\geq n_0$.

\noindent On the other hand, since $\lim\limits_{n\to\infty}\sigma_n=0$ and $\bar{\vartheta}<\frac{\bar{\beta}}{2}$, we also have
$\lim\limits_{n\to\infty}\left[(1-\sigma_n)(\bar{\beta}-2\bar{\vartheta})\right]=\bar{\beta}-2\bar{\vartheta}>0$. Thus, there exists $n_2\geq n_1$ such that $(1-\sigma_n)(\bar{\beta}-2\bar{\vartheta})>0~\forall n\geq n_2.$ \\
Similarly, there exists $n_3\geq n_2$ such that $(1-\sigma_n)(1-\bar{\vartheta})-\frac{\delta_{n-1}}{\delta_n}{\bar{r}}(1-\sigma_n)-\frac{1}{2}>0$ for all $n\geq n_3$.\\
 Using these facts in \eqref{LL**2*}, we obtain
\begin{eqnarray*}
\hat{t}_{n+1}&\leq & (1-\sigma_n)\hat{t}_n +\sigma_n\|\hat{v}-\widehat{w}\|^2~\forall n\geq n_3,
\end{eqnarray*}
which implies that $\{\hat{t}_n\}$ is bounded. It then follows from \eqref{Pos2*} that the sequence $\{w_n\}$ is indeed bounded, as claimed.
\end{proof}

\hfill

\begin{theorem}\label{THM2} Let $\{w_n\}$ be generated by Algorithm \ref{AL2} when Assumption \ref{AS1} holds. If $0\leq \bar{\vartheta}<\min\left\lbrace\frac{\bar{\beta}}{2},~\frac{\frac{1}{2}-\bar{r}}{2}\right\rbrace$, $\lim\limits_{n\to\infty}\sigma_n=0$ and $\sum\limits_{n=1}^\infty\sigma_n=\infty$, then $\{w_n\}$ converges strongly to $P_{(S+T)^{-1}(0)}\hat{v}$.
\end{theorem}
\begin{proof}
Let $\widehat{w}=P_{(S+T)^{-1}(0)}\hat{v}$. Then, using arguments similar to those used in the proof of Theorem \ref{THM1} up to \eqref{ABC}, we obtain
\begin{eqnarray}\label{ABC2a}
	&&\|w_{n+1}-\widehat{w}\|^2 +2\delta_{n}\langle Tw_{n+1}-Tw_n,\widehat{w}-w_{n+1}\rangle +\frac{1}{2}\|w_{n+1}-w_n\|^2 \nonumber\\
		&&\leq  (1-\sigma_n) \left(\|b_n-\widehat{w}\|^2 +2\delta_{n-1}\langle Tw_n-Tw_{n-1},\widehat{w}-w_{n}\rangle +\frac{1}{2}\|w_{n}-w_{n-1}\|^2\right) \nonumber\\
	&&+\sigma_n\left(\sigma_n\|\hat{v}-\widehat{w}\|^2+2(1-\sigma_n)\langle \hat{v}-\widehat{w}, b_n-\widehat{w}\rangle+2(1-\sigma_n)M \|w_{n+1}-b_n\|\right) \nonumber\\
	&&-(1-\sigma_n)^2\|w_{n+1}-b_n\|^2 +\left(\frac{\delta_{n-1}}{\delta_n}{\bar{r}}(1-\sigma_n)+\frac{1}{2}\right)\|w_{n+1}-w_n\|^2-\bar{\beta}(1-\sigma_n)\|w_n-w_{n-1}\|^2~\forall n\geq n_3.	
\end{eqnarray}	
Next, using \eqref{NL1} and \eqref{NL2} in \eqref{ABC2a}, we see that
\begin{eqnarray*}
	&&\|w_{n+1}-\widehat{w}\|^2 +2\delta_{n}\langle Tw_{n+1}-Tw_n,\widehat{w}-w_{n+1}\rangle +\frac{1}{2}\|w_{n+1}-w_n\|^2 \nonumber\\
		&&\leq  (1-\sigma_n) \left((1+\bar{\vartheta})\|w_n-\widehat{w}\|^2-\bar{\vartheta}\|w_{n-1}-\widehat{w}\|^2 +2\delta_{n-1}\langle Tw_n-Tw_{n-1},\widehat{w}-w_{n}\rangle +\frac{1}{2}\|w_{n}-w_{n-1}\|^2\right) \nonumber\\
	&&+(1-\sigma_n)\bar{\vartheta}(1+\bar{\vartheta})\|w_n-w_{n-1}\|^2+\sigma_n\left(\sigma_n\|\hat{v}-\widehat{w}\|^2+2(1-\sigma_n)\langle \hat{v}-\widehat{w}, b_n-\widehat{w}\rangle+2(1-\sigma_n)M \|b_n-w_{n+1}\|\right) \nonumber\\
	&&-(1-\sigma_n)^2\left[(1-\bar{\vartheta})\|w_{n+1}-w_n\|^2+(\bar{\vartheta}^2-\bar{\vartheta})\|w_n-w_{n-1}\|^2\right] \nonumber\\
	&&+\left(\frac{\delta_{n-1}}{\delta_n}{\bar{r}}(1-\sigma_n)+\frac{1}{2}\right)\|w_{n+1}-w_n\|^2-\bar{\beta}(1-\sigma_n)\|w_n-w_{n-1}\|^2~\forall n\geq n_3.	
\end{eqnarray*}	

\noindent This implies that
\begin{eqnarray}\label{ABC2}
	&&\|w_{n+1}-\widehat{w}\|^2-\bar{\vartheta}\|w_n-\widehat{w}\|^2 +2\delta_{n}\langle Tw_{n+1}-Tw_n,\widehat{w}-w_{n+1}\rangle +\frac{1}{2}\|w_{n+1}-w_n\|^2 \nonumber\\
		&&\leq  (1-\sigma_n) \left(\|w_n-\widehat{w}\|^2-\bar{\vartheta}\|w_{n-1}-\widehat{w}\|^2 +2\delta_{n-1}\langle Tw_n-Tw_{n-1},\widehat{w}-w_{n}\rangle +\frac{1}{2}\|w_{n}-w_{n-1}\|^2\right) \nonumber\\
	&&+\sigma_n\left(\sigma_n\|\hat{v}-\widehat{w}\|^2+2(1-\sigma_n)\langle \hat{v}-\widehat{w}, b_n-\widehat{w}\rangle+2(1-\sigma_n)M \|b_n-w_{n+1}\|\right) \nonumber\\
		&&-\left((1-\sigma_n)^2(1-\bar{\vartheta})-\frac{\delta_{n-1}}{\delta_n}{\bar{r}}(1-\sigma_n)-\frac{1}{2}\right)\|w_{n+1}-w_n\|^2\nonumber\\
	&&-(1-\sigma_n)\left(\bar{\beta}-2\bar{\vartheta}-\sigma_n(\bar{\vartheta}^2-\bar{\vartheta})\right)\|w_n-w_{n-1}\|^2~\forall n\geq n_3.	
\end{eqnarray}	

\noindent Therefore, for all $n\geq n_3$, we find that
\begin{eqnarray}\label{STR}
\hat{t}_{n+1} &\leq & (1-\sigma_n)\hat{t}_n+\sigma_n \hat{h}_n,
\end{eqnarray}
where $\hat{h}_n=\sigma_n\|\hat{v}-\widehat{w}\|^2+2(1-\sigma_n)\langle \hat{v}-\widehat{w}, b_n-\widehat{w}\rangle+2(1-\sigma_n)M \|w_{n+1}-b_n\|$.

\noindent As in the proof of Theorem \ref{THM1}, let $\{\hat{t}_{n_i}\}$ be a subsequence of $\{\hat{t}_n\}$
such that $\liminf\limits_{i\to \infty}\left(\hat{t}_{n_i+1}-\hat{t}_{n_i}\right)\geq 0$.
Then it follows from \eqref{ABC2} that
\begin{eqnarray*}
&&\limsup_{i\to\infty}\left[\left((1-\sigma_{n_i})^2(1-\bar{\vartheta})-\frac{\delta_{{n_i}-1}}{\delta_{n_i}}{\bar{r}}(1-\sigma_{n_i})-\frac{1}{2}\right)\|w_{{{n_i}}+1}-w_{{n_i}}\|^2\right]\\
&&\leq \limsup_{i\to \infty}\left[(\hat{t}_{n_i}-\hat{t}_{{n_i}+1})+\sigma_{n_i} (\hat{h}_{n_i}-\hat{t}_{n_i})\right]\\
&&\leq -\liminf_{i\to \infty}(\hat{t}_{{n_i}+1}-\hat{t}_{n_i})\leq 0.
\end{eqnarray*}
Since $\lim\limits_{i\to\infty}\left((1-\sigma_{n_i})^2(1-\bar{\vartheta})-\frac{\delta_{{n_i}-1}}{\delta_{n_i}}{\bar{r}}(1-\sigma_{n_i})-\frac{1}{2}\right)=\frac{1}{2}-\bar{\vartheta}-\bar{r}>0$,
we get \begin{eqnarray}\label{Z12}
\lim_{i\to\infty}\|w_{{n_i}+1}-w_{n_i}\|=0.
\end{eqnarray}
Thus,
\begin{eqnarray}\label{Z12*}
\lim_{i\to\infty}\|b_{{n_i}}-w_{n_i}\|=\lim_{i\to\infty}\bar{\vartheta}\|w_{n_i}-w_{n_i-1}\|=0.
\end{eqnarray}
Using \eqref{Z12} and \eqref{Z12*}, we obtain
\begin{eqnarray}\label{Z22*}
\lim_{i\to\infty} \|w_{n_i+1}-b_{n_i}\|=0.
\end{eqnarray}
As in the proof of Theorem \ref{THM1}, we can choose a subsequence $\{w_{n_{i_j}}\}$ of $\{w_{n_i}\}$ which converges weakly to some $w^*\in \mathcal{H}$ such that
\begin{eqnarray*}
\limsup_{i\to\infty}\langle \hat{v}-\widehat{w}, w_{n_i}-\widehat{w}\rangle=\lim_{j\to\infty}\langle \hat{v}-\widehat{w}, w_{{n_i}_j}-\widehat{w}\rangle=\langle \hat{v}-\widehat{w}, w^*-\widehat{w}\rangle,
\end{eqnarray*}
and we can show that $w^*\in (S+T)^{-1}(0)$. Since $\widehat{w}=P_{(S+T)^{-1}(0)}\hat{v}$, we have
\begin{eqnarray*}
\limsup_{i\to\infty}\langle \hat{v}-\widehat{w}, w_{n_i}-\widehat{w}\rangle=\langle \hat{v}-\widehat{w}, w^*-\widehat{w}\rangle\leq 0,
\end{eqnarray*}
which implies by \eqref{Z12*} that
\begin{eqnarray}\label{Z6}
\limsup_{i\to\infty}\langle \hat{v}-\widehat{w}, b_{n_i}-\widehat{w}\rangle\leq 0.
\end{eqnarray}
Using \eqref{Z22*}, \eqref{Z6} and the condition $\lim\limits_{i\to\infty}\sigma_{n_{i}}=0$, we get that $\limsup\limits_{i\to\infty}\hat{h}_{n_{i}}\leq0$. Thus, applying Lemma \ref{lem6} to \eqref{STR}, we obtain $\lim\limits_{n\to\infty}\hat{t}_n=0$. Hence, using \eqref{Pos2*}, we conclude that $\{w_n\}$ converges strongly to $\widehat{w}=P_{(S+T)^{-1}(0)}\hat{v}$,
as asserted.
\end{proof}

\hfill

\begin{remark}
Instead of using a constant inertial parameter $\bar{\vartheta}$ in Algorithm \ref{AL2}, we can use a variable inertial parameter $\vartheta_n$, where $0\leq \vartheta_n\leq \vartheta_{n+1}  \leq \bar{\vartheta}$, without imposing an additional assumption on the anchoring coefficient $\sigma_n$ as done in most related works \cite{Timi, AAMPrasit, YNUMA, YOMS}, where additional requirements are imposed on $\sigma_n$.

\noindent At this point we note that although the sequence $\{\vartheta_n\}$ is required to be increasing, unlike in most related papers \cite{Timi, FBFS4, AAMPrasit, Sahu, YOMS}, it does not depend on the iterates $\{w_n\}$ and $\{w_{n-1}\}$.

\hfill

\noindent We now present our algorithm with variable inertial parameter and the corresponding strong convergence theorem.
\end{remark}
\hrule \hrule
\begin{alg}\label{AL3} Let $\delta_0,\delta_1>0$, $\bar{r}\in \Big(\bar{\beta},~\frac{1-2\bar{\beta}}{2}\Big)$ with $\bar{\beta}\in (0, \frac{1}{4})$, and choose the sequences $\{\sigma_n\}$ in $(0, 1)$ and $\{c_n\}$ in $[0, \infty)$  such that $\sum_{n=1}^{\infty} c_n < \infty$. Given arbitrary $\hat{v}, w_0,w_1 \in \mathcal{H}$, let the sequence $\{w_n\}$ be generated by
	\hrule \hrule
	 \begin{eqnarray*}
w_{n+1}=J^{S}_{\delta_n}\left(\sigma_n \hat{v}+(1-\sigma_n)(w_n+\vartheta_n(w_n-w_{n-1}))-\delta_nTw_n-\delta_{n-1}(1-\sigma_n)(Tw_{n}-Tw_{n-1})\right), ~\forall n\geq 1,
\end{eqnarray*}
where the step size $\delta_n$ is defined by \eqref{a11}.
\hrule\hrule
\end{alg}

\hfill

\begin{theorem}\label{THM3} Let the sequence $\{w_n\}$ be generated by Algorithm \ref{AL3} when Assumption \ref{AS1} holds. If $0\leq  \vartheta_n\leq \vartheta_{n+1} \leq \bar{\vartheta}<\min\left\lbrace\frac{\bar{\beta}}{2},~\frac{\frac{1}{2}-\bar{r}}{2}\right\rbrace$, $\lim\limits_{n\to\infty}\sigma_n=0$ and $\sum\limits_{n=1}^\infty\sigma_n=\infty$, then $\{w_n\}$ converges strongly to $P_{(S+T)^{-1}(0)}\hat{v}$.
\end{theorem}
\begin{proof}
Using $\vartheta_n$ instead of $\bar{\vartheta}$ in the proof of Lemma \ref{lem2} up to \eqref{LL**2*}, we have
\begin{eqnarray*}
	&&\|w_{n+1}-\widehat{w}\|^2-\vartheta_n\|w_n-\widehat{w}\|^2 +2\delta_{n}\langle Tw_{n+1}-Tw_n,\widehat{w}-w_{n+1}\rangle +\frac{1}{2}\|w_{n+1}-w_n\|^2 \nonumber\\
	&&\leq  (1-\sigma_n)\left[\|w_n-\widehat{w}\|^2-\vartheta_n\|w_{n-1}-\widehat{w}\|^2 +2\delta_{n-1}\langle Tw_n-Tw_{n-1},\widehat{w}-w_{n}\rangle +\frac{1}{2}\|w_n-w_{n-1}\|^2\right] \nonumber\\
&&+\sigma_n\|\hat{v}-\widehat{w}\|^2-(1-\sigma_n)(\bar{\beta}-2\vartheta_n)\|w_n-w_{n-1}\|^2\nonumber\\
&&-\left[(1-\sigma_n)(1-\vartheta_n)-\frac{\delta_{n-1}}{\delta_n}{\bar{r}}(1-\sigma_n)-\frac{1}{2}\right]\|w_{n+1}-w_n\|^2~\forall n\geq n_0.
\end{eqnarray*}

\noindent Since $\vartheta_n\leq \vartheta_{n+1}$, we obtain
\begin{eqnarray*}
	&&\|w_{n+1}-\widehat{w}\|^2-\vartheta_{n+1}\|w_n-\widehat{w}\|^2 +2\delta_{n}\langle Tw_{n+1}-Tw_n,\widehat{w}-w_{n+1}\rangle +\frac{1}{2}\|w_{n+1}-w_n\|^2 \nonumber\\
	&&\leq  (1-\sigma_n)\left[\|w_n-\widehat{w}\|^2-\vartheta_n\|w_{n-1}-\widehat{w}\|^2 +2\delta_{n-1}\langle Tw_n-Tw_{n-1},\widehat{w}-w_{n}\rangle +\frac{1}{2}\|w_n-w_{n-1}\|^2\right] \nonumber\\
&&+\sigma_n\|\hat{v}-\widehat{w}\|^2-(1-\sigma_n)(\bar{\beta}-2\vartheta_n)\|w_n-w_{n-1}\|^2\nonumber\\
&&-\left[(1-\sigma_n)(1-\vartheta_n)-\frac{\delta_{n-1}}{\delta_n}{\bar{r}}(1-\sigma_n)-\frac{1}{2}\right]\|w_{n+1}-w_n\|^2~\forall n\geq n_0.
\end{eqnarray*}
Thus, by setting $\hat{t}_n:=\|w_n-\widehat{w}\|^2-\vartheta_n\|w_{n-1}-\widehat{w}\|^2 +2\delta_{n-1}\langle Tw_n-Tw_{n-1},\widehat{w}-w_{n}\rangle +\frac{1}{2}\|w_n-w_{n-1}\|^2$, and following the arguments used in obtaining \eqref{Pos2*}, we obtain
\begin{eqnarray}\label{Pos2*VS}
\hat{t}_n&\geq  & \left(\frac{1}{2}-\frac{\delta_{n-1}}{\delta_{n}} \bar{r}-2\vartheta_n\right)\left[\|w_n-\widehat{w}\|^2+\|w_n-w_{n-1}\|^2\right].
\end{eqnarray}
Since $\vartheta_n\leq \bar{\vartheta}<\frac{\frac{1}{2}-\bar{r}}{2}$, we get that
$\lim\limits_{n\to\infty}\left(\frac{1}{2}-\frac{\delta_{n-1}}{\delta_{n}} \bar{r}-2\vartheta_n\right)\geq\lim\limits_{n\to\infty}\left(\frac{1}{2}-\frac{\delta_{n-1}}{\delta_{n}} \bar{r}-2\bar{\vartheta}\right)=\frac{1}{2}-\bar{r}-2\bar{\vartheta}>0$. Thus, there exists $n_1\in \mathbb{N}$, $n_1\geq n_0$ such that $\frac{1}{2}-\frac{\delta_{n-1}}{\delta_{n}} \bar{r}-2\vartheta_n>0~\forall n\geq n_1.$ Hence, we get from \eqref{Pos2*VS} that $\hat{t}_n\geq 0$ for all $n\geq n_1\geq n_0$, and by following the arguments from \eqref{Pos2*} up to the end of the proof of Lemma \ref{lem2},  we will see that $\{w_n\}$ is bounded.\\
 Now, using arguments similar to those used in the proof of Theorem \ref{THM2}, and  with the condition $\vartheta_n\leq \vartheta_{n+1} \leq \bar{\vartheta}<\min\left\lbrace\frac{\bar{\beta}}{2},~\frac{\frac{1}{2}-\bar{r}}{2}\right\rbrace$ in mind, we will get that $\{w_n\}$ converges strongly to $P_{(S+T)^{-1}(0)}\hat{v}$, as asserted.
\end{proof}

\section{Viscosity-type Forward-Reflected-Backward Splitting Method}\label{VIS}
\noindent Let $U$ be a contraction with constant $\bar{\kappa}\in [0, \frac{1}{2})$. We propose the following viscosity variant of Algorithm \ref{AL1}.

\hrule \hrule
\begin{alg}\label{AL1V} Let $\delta_0,\delta_1>0$, $\bar{r}\in \Big(\bar{\beta},~\frac{1-2\bar{\beta}}{2}\Big)$ with $\bar{\beta}\in (0, \frac{1}{4})$, and choose the sequences $\{\sigma_n\}$ in $(0, 1)$ and $\{c_n\}$ in $[0, \infty)$ such that $\sum_{n=1}^{\infty} c_n< \infty$. For arbitrary $w_0,w_1 \in \mathcal{H}$, let the sequence $\{w_n\}$ be generated by
	\hrule \hrule
	 \begin{eqnarray*}
w_{n+1}=J^{S}_{\delta_n}\big(\sigma_n Uw_n+(1-\sigma_n) w_n-\delta_nTw_n-\delta_{n-1}(1-\sigma_n(1-2\bar{\kappa}))(Tw_{n}-Tw_{n-1})\big),~ n\geq 1,
\end{eqnarray*}
where the step size $\delta_n$ is defined by \eqref{a11}.
\hrule\hrule
\end{alg}

\hfill

\begin{lemma}\label{lem2V} Let $\{w_n\}$ be generated by Algorithm \ref{AL1V} when Assumption \ref{AS1} holds. If $\lim\limits_{n\to\infty}\sigma_n=0$, then the sequence $\{w_n\}$ is bounded.
\end{lemma}

\begin{proof}
Let $\widehat{w}\in (S+T)^{-1}(0)$ and set $a_n:=\sigma_n Uw_n+(1-\sigma_n) w_n$. Then by similar arguments as in the proof of Lemma \ref{lem1} up to \eqref{LL*}, we get
\begin{eqnarray}\label{LL*V}
	&&\|w_{n+1}-\widehat{w}\|^2 +2\delta_{n}\langle Tw_{n+1}-Tw_n,\widehat{w}-w_{n+1}\rangle +\frac{1}{2}\|w_{n+1}-w_n\|^2 \nonumber\\
	&&\leq  \|a_n-\widehat{w}\|^2-\|w_{n+1}-a_n\|^2 +2\delta_{n-1}(1-\sigma_n(1-2\bar{\kappa}))\langle Tw_n-Tw_{n-1},\widehat{w}-w_{n}\rangle  \nonumber\\
&&+(1-\sigma_n(1-2\bar{\kappa}))\big(\frac{1}{2}-\bar{\beta}\big)\|w_n-w_{n-1}\|^2+\left[\frac{\delta_{n-1}}{\delta_n}{\bar{r}}(1-\sigma_n(1-2\bar{\kappa}))+\frac{1}{2}\right]\|w_{n+1}-w_n\|^2~\forall n\geq n_0.	
\end{eqnarray}	
Replacing $\hat{v}$ by $Uw_n$ in \eqref{H1} and \eqref{H2}, we get
\begin{eqnarray}\label{H1V}
\|a_n-\widehat{w}\|^2&=&\|w_n-\widehat{w}\|^2+\sigma_n^2\|w_n-Uw_n\|^2-\sigma_n\|w_n-Uw_n\|^2-\sigma_n\|w_n-\widehat{w}\|^2+\sigma_n\|Uw_n-\widehat{w}\|^2
\end{eqnarray}
and
\begin{eqnarray}\label{H2V}
\|a_n-w_{n+1}\|^2&=&\|w_n-w_{n+1}\|^2+\sigma_n^2\|w_n-Uw_n\|^2-\sigma_n\|w_n-Uw_n\|^2-\sigma_n\|w_n-w_{n+1}\|^2\nonumber \\
&&+\sigma_n\|Uw_n-w_{n+1}\|^2,
\end{eqnarray}
respectively.

\noindent Now, subtracting \eqref{H2V} from \eqref{H2V}, we obtain
\begin{eqnarray}\label{H**V}
&&\|a_n-\widehat{w}\|^2-\|a_n-w_{n+1}\|^2\nonumber\\
&=&(1-\sigma_n)\|w_n-\widehat{w}\|^2+\sigma_n\|Uw_n-\widehat{w}\|^2-(1-\sigma_n)\|w_{n+1}-w_n\|^2-\sigma_n\|w_{n+1}-Uw_n\|^2\nonumber\\
&\leq & (1-\sigma_n)\|w_n-\widehat{w}\|^2+2\sigma_n\|Uw_n-U\widehat{w}\|^2+2\sigma_n\|U\widehat{w}-\widehat{w}\|^2-(1-\sigma_n)\|w_{n+1}-w_n\|^2\nonumber\\
&\leq & (1-\sigma_n)\|w_n-\widehat{w}\|^2+2\sigma_n\bar{\kappa}^2\|w_n-\widehat{w}\|^2+2\sigma_n\|U\widehat{w}-\widehat{w}\|^2-(1-\sigma_n)\|w_{n+1}-w_n\|^2\nonumber\\
&\leq & (1-\sigma_n(1-2\bar{\kappa}))\|w_n-\widehat{w}\|^2+\sigma_n(1-2\bar{\kappa})\frac{2\|U\widehat{w}-\widehat{w}\|^2}{1-2\bar{\kappa}}-(1-\sigma_n)\|w_{n+1}-w_n\|^2.
\end{eqnarray}

\noindent Using \eqref{H**V} in \eqref{LL*V}, we get
\begin{eqnarray}\label{LL**V}
	&&\|w_{n+1}-\widehat{w}\|^2 +2\delta_{n}\langle Tw_{n+1}-Tw_n,\widehat{w}-w_{n+1}\rangle +\frac{1}{2}\|w_{n+1}-w_n\|^2 \nonumber\\
	&&\leq  (1-\sigma_n(1-2\bar{\kappa}))\|w_n-\widehat{w}\|^2+\sigma_n(1-2\bar{\kappa})\frac{2\|U\widehat{w}-\widehat{w}\|^2}{1-2\bar{\kappa}}-(1-\sigma_n)\|w_{n+1}-w_n\|^2 \nonumber\\
	&& +2\delta_{n-1}(1-\sigma_n(1-2\bar{\kappa}))\langle Tw_n-Tw_{n-1},\widehat{w}-w_{n}\rangle +(1-\sigma_n(1-2\bar{\kappa}))\big(\frac{1}{2}-\bar{\beta}\big)\|w_n-w_{n-1}\|^2 \nonumber\\
&&+\left[\frac{\delta_{n-1}}{\delta_n}{\bar{r}}(1-\sigma_n(1-2\bar{\kappa}))+\frac{1}{2}\right]\|w_{n+1}-w_n\|^2\nonumber\\
&&=  (1-\sigma_n(1-2\bar{\kappa}))\left[\|w_n-\widehat{w}\|^2 +2\delta_{n-1}\langle Tw_n-Tw_{n-1},\widehat{w}-w_{n}\rangle +\frac{1}{2}\|w_n-w_{n-1}\|^2\right] \nonumber\\
&&+\sigma_n(1-2\bar{\kappa})\frac{2\|U\widehat{w}-\widehat{w}\|^2}{1-2\bar{\kappa}} -\bar{\beta}(1-\sigma_n(1-2\bar{\kappa}))\|w_{n}-w_{n-1}\|^2-(1-\sigma_n)\|w_{n+1}-w_{n}\|^2\nonumber\\
&&+\left[\frac{\delta_{n-1}}{\delta_n}\bar{r}(1-\sigma_n(1-2\bar{\kappa}))+\frac{1}{2}\right]\|w_{n+1}-w_n\|^2 \nonumber\\
&&=  (1-\sigma_n(1-2\bar{\kappa}))\left[\|w_n-\widehat{w}\|^2 +2\delta_{n-1}\langle Tw_n-Tw_{n-1},\widehat{w}-w_{n}\rangle +\frac{1}{2}\|w_n-w_{n-1}\|^2\right] \nonumber\\
&&+\sigma_n(1-2\bar{\kappa})\frac{2\|U\widehat{w}-\widehat{w}\|^2}{1-2\bar{\kappa}}-\left[\frac{1}{2}-\sigma_n-\frac{\delta_{n-1}}{\delta_n}\bar{r}(1-\sigma_n(1-2\bar{\kappa}))\right]\|w_{n+1}-w_n\|^2\nonumber\\
&& -\bar{\beta}(1-\sigma_n(1-2\bar{\kappa}))\|w_{n}-w_{n-1}\|^2~\forall n\geq n_0.	
\end{eqnarray}	
\noindent Let $t_n:=\|w_n-\widehat{w}\|^2 +2\delta_{n-1}\langle Tw_n-Tw_{n-1},\widehat{w}-w_{n}\rangle +\frac{1}{2}\|w_n-w_{n-1}\|^2$. Then by \eqref{Pos}, $t_n\geq 0$ for all $n\geq n_0$. We also have that
$$\lim\limits_{n\to\infty}\left[\frac{1}{2}-\sigma_n-\frac{\delta_{n-1}}{\delta_n}{\bar{r}}(1-\sigma_n(1-2\bar{\kappa}))\right]=\frac{1}{2}-\bar{r}>\bar{\beta}.$$
Thus, there exists $n_1\in \mathbb{N}$, $n_1\geq n_0$, such that $\frac{1}{2}-\sigma_n-\frac{\delta_{n-1}}{\delta_n}{\bar{r}}(1-\sigma_n(1-2\bar{\kappa}))>\bar{\beta}~\forall n\geq n_1.$ Hence,
\begin{eqnarray*}
t_{n+1}&\leq & (1-\sigma_n(1-2\bar{\kappa}))t_n +\sigma_n(1-2\bar{\kappa})\frac{2\|U\widehat{w}-\widehat{w}\|^2}{1-2\bar{\kappa}}-\bar{\beta}\|w_{n+1}-w_n\|^2\nonumber\\
&&-\bar{\beta}(1-\sigma_n(1-2\bar{\kappa}))\|w_n-w_{n-1}\|^2 ~\forall n\geq n_1\geq n_0.
\end{eqnarray*}
Therefore, $\{t_n\}$ is bounded which further gives that $\{w_n\}$ is bounded too, as asserted.
\end{proof}

\hfill

\begin{theorem}\label{THM1V} Let $\{w_n\}$ be generated by Algorithm \ref{AL1V} when Assumption \ref{AS1} holds. If $U$ is a contraction with $\bar{\kappa}\in [0, \frac{1}{2})$, $\lim\limits_{n\to\infty}\sigma_n=0$ and $\sum\limits_{n=1}^\infty\sigma_n=\infty$, then $\{w_n\}$ converges strongly to $P_{(S+T)^{-1}(0)}U\widehat{w}$.
\end{theorem}
\begin{proof}
Let $\widehat{w}=P_{(S+T)^{-1}(0)}U\widehat{w}$. Replacing $\hat{v}$ by $U w_n$ in \eqref{T1}, we obtain
 \begin{eqnarray}\label{T1V}
\|a_n-\widehat{w}\|^2&=&\sigma_n^2\|Uw_n-\widehat{w}\|^2+(1-\sigma_n)^2\|w_n-\widehat{w}\|^2+2\sigma_n(1-\sigma_n)\langle Uw_n-\widehat{w}, w_n-\widehat{w}\rangle\nonumber\\
&=&(1-\sigma_n)^2\|w_n-\widehat{w}\|^2+\sigma_n^2\|Uw_n-\widehat{w}\|^2+2\sigma_n(1-\sigma_n)\langle Uw_n-U\widehat{w}, w_n-\widehat{w}\rangle\nonumber\\
&&+2\sigma_n(1-\sigma_n)\langle U\widehat{w}-\widehat{w}, w_n-\widehat{w}\rangle\nonumber\\
&\leq&(1-\sigma_n)^2\|w_n-\widehat{w}\|^2+\sigma_n^2\|Uw_n-\widehat{w}\|^2+2\sigma_n(1-\sigma_n)\bar{\kappa}\|w_n-\widehat{w}\|^2\nonumber\\
&&+2\sigma_n(1-\sigma_n)\langle U\widehat{w}-\widehat{w}, w_n-\widehat{w}\rangle\nonumber\\
&\leq &(1-\sigma_n(1-2\bar{\kappa}))\|w_n-\widehat{w}\|^2+\sigma_n^2\|w_n-\widehat{w}\|^2+\sigma_n^2\|Uw_n-\widehat{w}\|^2\nonumber\\
&&+2\sigma_n(1-\sigma_n)\langle U\widehat{w}-\widehat{w}, w_n-\widehat{w}\rangle.
\end{eqnarray}
Also, replacing $\hat{v}$ by $Uw_n$ in \eqref{T2}, we obtain
\begin{eqnarray}\label{T2V}
\|a_n-w_{n+1}\|^2&\geq &\sigma_n^2\|w_{n+1}-Uw_n\|^2+(1-\sigma_n)^2\|w_{n+1}-w_n\|^2-2\sigma_n(1-\sigma_n)\widehat{M} \|w_{n+1}-w_n\|,
\end{eqnarray}
where $\widehat{M}:=\sup\limits_{n\geq 1}\|w_{n+1}-Uw_n\|$.

\noindent Now, using \eqref{T1V} and \eqref{T2V} in \eqref{LL*V}, we get
\begin{eqnarray*}
	&&\|w_{n+1}-\widehat{w}\|^2 +2\delta_{n}\langle Tw_{n+1}-Tw_n,\widehat{w}-w_{n+1}\rangle +\frac{1}{2}\|w_{n+1}-w_n\|^2 \nonumber\\
	&&\leq (1-\sigma_n(1-2\bar{\kappa})) \|w_n-\widehat{w}\|^2+\sigma_n^2\|w_n-\widehat{w}\|^2+\sigma_n^2\|Uw_n-\widehat{w}\|^2+2\sigma_n(1-\sigma_n)\langle U\widehat{w}-\widehat{w}, w_n-\widehat{w}\rangle\nonumber\\
	&&-\sigma_n^2\|w_{n+1}-Uw_n\|^2-(1-\sigma_n)^2\|w_{n+1}-w_n\|^2+2\sigma_n(1-\sigma_n)\widehat{M}\|w_{n+1}-w_n\|^2\nonumber\\
	&&+2\delta_{n-1}(1-\sigma_n(1-2\bar{\kappa}))\langle Tw_n-Tw_{n-1},\widehat{w}-w_{n}\rangle +(1-\sigma_n(1-2\bar{\kappa}))\big(\frac{1}{2}-\bar{\beta}\big)\|w_n-w_{n-1}\|^2\nonumber\\
	&&+\left[\frac{\delta_{n-1}}{\delta_n}{\bar{r}}(1-\sigma_n(1-2\bar{\kappa}))+\frac{1}{2}\right]\|w_{n+1}-w_n\|^2	\nonumber\\
	&&\leq (1-\sigma_n(1-2\bar{\kappa})) \left(\|w_n-\widehat{w}\|^2+2\delta_{n-1}\langle Tw_n-Tw_{n-1},\widehat{w}-w_{n}\rangle +\frac{1}{2}\|w_n-w_{n-1}\|^2\right)\nonumber\\
	&&+\sigma_n \left(\sigma_n\|w_n-\widehat{w}\|^2+\sigma_n\|Uw_n-\widehat{w}\|^2+2(1-\sigma_n)\langle U\widehat{w}-\widehat{w}, w_n-\widehat{w}\rangle+2(1-\sigma_n)\widehat{M}\|w_{n+1}-w_n\|^2\right)\nonumber\\
	&& -\bar{\beta}(1-\sigma_n(1-2\bar{\kappa}))\|w_n-w_{n-1}\|^2-\left[\frac{1}{2}+\sigma_n^2-2\sigma_n-\frac{\delta_{n-1}}{\delta_n}{\bar{r}}(1-\sigma_n(1-2\bar{\kappa}))\right]\|w_{n+1}-w_n\|^2~\forall n\geq n_1.	
\end{eqnarray*}	
Therefore, for all $n\geq n_1$, we have
$$t_{n+1}\leq (1-\sigma_n(1-2\bar{\kappa}))t_n+\sigma_n(1-2\bar{\kappa})\widetilde{h}_n,$$
where\\
 $\widetilde{h}_n=(1-2\bar{\kappa})^{-1}\left(\sigma_n\|w_n-\widehat{w}\|^2+\sigma_n\|Uw_n-\widehat{w}\|^2+2(1-\sigma_n)\langle U\widehat{w}-\widehat{w}, w_n-\widehat{w}\rangle+2(1-\sigma_n)\widehat{M}\|w_{n+1}-w_n\|^2\right)$.

 \noindent Thus, by arguments similar to those from \eqref{STRa} to the end of the proof of Theorem \ref{THM1}, we see that $\{w_n\}$ converges strongly to $\widehat{w}=P_{(S+T)^{-1}(0)}U\widehat{w}$, as asserted.
\end{proof}

\hfill

\noindent By incorporating the inertial extrapolation step into Algorithm \ref{AL1V}, we arrive at the following inertial variant of Algorithm \ref{AL1V} or the viscosity variant of Algorithm \ref{AL2}, namely the  {\it inertial viscosity-type forward-reflected-backward splitting method}.

\hrule \hrule
\begin{alg}\label{AL2V} Let $\delta_0,\delta_1>0$, $\bar{\vartheta}\in [0, 1)$, $\bar{r}\in \Big(\bar{\beta},~\frac{1-2\bar{\beta}}{2}\Big)$ with $\bar{\beta}\in (0, \frac{1}{4})$, and choose sequences $\{\sigma_n\}$ in $(0, 1)$ and $\{c_n\}$ in $[0, \infty)$ such that $\sum_{n=1}^{\infty} c_n< \infty$. For arbitrary $w_0,w_1 \in \mathcal{H}$, let the sequence $\{w_n\}$ be generated by
	\hrule \hrule
	 \begin{eqnarray*}
w_{n+1}=J^{S}_{\delta_n}\left(\sigma_n Uw_n +(1-\sigma_n)(w_n+\bar{\vartheta}(w_n-w_{n-1}))-\delta_nTw_n-\delta_{n-1}(1-\sigma_n(1-2\bar{\kappa}))(Tw_{n}-Tw_{n-1})\right),~\forall n\geq 1,
\end{eqnarray*}
where the step size $\delta_n$ is defined by \eqref{a11}.
\hrule\hrule
\end{alg}
\noindent Combining Theorem \ref{THM2} and Theorem \ref{THM1V}, we arrive at the following theorem for Algorithm \ref{AL2V}.
\begin{theorem}\label{THM2V} Let $\{w_n\}$ be generated by Algorithm \ref{AL2V} when Assumption \ref{AS1} holds. If $U$ is a contraction with constant $\bar{\kappa}\in [0, \frac{1}{2})$, $0\leq \bar{\vartheta}<\min\left\lbrace\frac{\bar{\beta}}{2},~\frac{\frac{1}{2}-\bar{r}}{2}\right\rbrace$, $\lim\limits_{n\to\infty}\sigma_n=0$ and $\sum\limits_{n=1}^\infty\sigma_n=\infty$, then $\{w_n\}$ converges strongly to $P_{(S+T)^{-1}(0)}U\widehat{w}$.
\end{theorem}

\hfill

\begin{remark} Like in Algorithm \ref{AL3}, we can replace the constant inertial parameter $\bar{\vartheta}$ in Algorithm \ref{AL2V} with an increasing sequence of variable inertial parameters $\{\vartheta_n\}$
without imposing an additional condition on the coefficient $\{\sigma_n\}$.\\
\end{remark}

\begin{remark}
In all the algorithms we proposed in this paper, we obtain strong convergence results for the monotone inclusion problem \eqref{MIP} without assuming that either the maximal monotone operator $S$ or the Lipschitz monotone operator $T$ are strongly monotone (a condition that is quite restrictive). Rather, we modify the forward-reflected-backward splitting method in \cite{MT} appropriately in order to obtain our results.
\end{remark}

\hfill

\section{Numerical Illustrations}\label{Se5}

\noindent In this section we provide computational experiments and compare Algorithm \ref{AL1} and Algorithm \ref{AL2} with the forward-reflected-backward splitting method (FRBSM) \cite[Algorithm (2.2)]{MT} and the reflected-forward-backward splitting method (RFBSM) \cite[Algorithm (1.6)]{MT3}. We use test examples which originate in image restoration and optimal control problems, as well as academic examples.

\hfill

\subsection{Image Restoration Problem}

\begin{example}\label{ex2*}
\begin{eqnarray}\label{IRP}
\min_{w\in \mathcal{R}^l} ~\{||\mathcal{P}w-e||^2_2 +\bar{\delta} ||w||_1\},
\end{eqnarray}
where $\bar{\delta}>0$ (in particular, we take $\bar{\delta}=1$), $w\in \mathcal{R}^l$ is the original image that we intend to recover, $e\in \mathcal{R}^d$ is the observed image and $\mathcal{P}: \mathcal{R}^l\to \mathcal{R}^d$ is the blurring operator. For the numerical computation, we use the $205 \times 232$ Tire Image found in the MATLAB Image Processing Toolbox. Also, we use the Gaussian blur of size $9\times 9$ and standard deviation $\sigma=4$ to create the blurred and noisy image (observed image). In order to measure the quality of the restored image, we use  the signal-to-noise ratio which is defined by
\begin{equation}
	\mbox{SNR} = 20 \times \log_{10}\left(\frac{\|w\|_2}{\|w-w^*\|_2}\right), \nonumber
\end{equation}
where $w^*$ is the restored image.
Note that the larger the SNR, the better the quality of the restored image. \\\ For this example, we choose $w_0 = \textbf{0} \in \mathcal{R}^{l\times l}$ and $w_1 = \textbf{1} \in \mathcal{R}^{l\times d}.$ \\
Furthermore, we choose $\delta_0 = 0.01, \delta_1 = 0.3, \bar{r}=0.3, \hat{v} = 2*\textbf{1}, \sigma_n=\frac{1}{n+250}, c_n = \frac{1}{(n+100)^2}, \bar{\vartheta} = 0.0000005$ for Algorithms \ref{AL1} and \ref{AL2} while we  take $\lambda_n = \frac{2n+1}{111n+100}$ and $\gamma = 0.01$ for FRBSM \cite[Algorithm (2.2)]{MT} and RFBSM \cite[Algorithm (1.6)]{MT3}, respectively. The computational results are shown in Table \ref{im1}, Figures \ref{figim} and \ref{im2}.

 \begin{figure}
	\begin{center}
		\includegraphics[width=12cm]{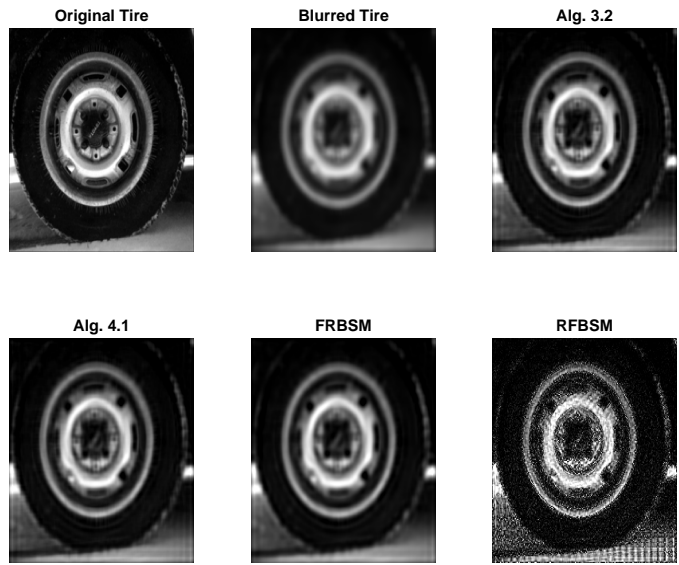}	
	\end{center}
	\caption{Numerical results for {\bf Example \ref{ex2*}:} Top Left: original image; Top Middle: blurred image; Top Right:  restored image by Algorithm \ref{AL1} with SNR = 24.7440; Bottom Left: restored image by Algorithm \ref{AL2} with SNR = 24.7440; Bottom Middle: restored image by FRBSM with SNR = 24.4121; Bottom Right: restored image by RFBSM with SNR = 16.6445.}\label{figim}
\end{figure}
\noindent

\begin{table}[h]
	\caption{{\bf Example \ref{ex2*}:} Numerical comparison of Algorithm \ref{AL1}, Algorithm \ref{AL2}, FRBSM and RFBSM using their SNR values.}
	\label{im1}
	\begin{tabular}{ p{2.5cm} p{2.5cm} p{2.5cm} p{2.5cm} p{2.5cm} p{2.5cm}}
		\hline
		Images & $n$ & Alg. \ref{AL1} & Alg. \ref{AL2} & FRBSM & RFBSM \\
		\hline
		Tire.tif   &   200 & 24.4142 & 24.4142 & 20.1744 & 13.8430 \\
		($205\times 232$)& 500 & 24.6659  & 24.6659 & 22.8918 & 16.5377\\
		& 1000 & 24.7202 & 24.7202 & 23.8513 & 16.7678	\\
		& 1500 & 24.7440 & 24.7440 & 24.4121 & 16.6445\\
		\hline
		\end{tabular}
\end{table}
\begin{figure}
	\begin{center}
		\includegraphics[width=8cm]{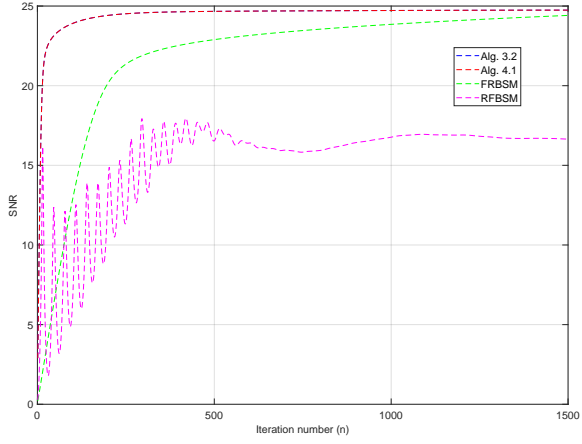}	
	\end{center}
	\caption{Numerical results for {\bf Example \ref{ex2*}:} Showing different SNR for each algorithms.}\label{im2}
\end{figure}
\end{example}

\hfill

\hfill

\subsection{Optimal Control Problem}${}$
\begin{example}\label{OCP}
\noindent  Let $\mathfrak{L}^2([0,\mathsf{T}],\mathcal{R}^l)$
be the Hilbert space of all square integrable, measurable vector functions
$z:[0,\mathsf{T}]\to\mathcal{R}^l$, where $0 < \mathsf{T} \in \mathcal{R}$.

\noindent We consider the optimal control problem
\begin{equation}\label{OCp}
    z^*(t)=\text{argmin}\{\phi(z):z\in \widetilde{Z}\}
\end{equation}
on the interval $[0,\mathsf{T}]$, where $\widetilde{Z}$ is the set of admissible controls and consists of continuous functions, that is,
$$\widetilde{Z}=\big\{z(t)\in \mathfrak{L}^2([0,\mathsf{T}],\mathcal{R}^l):z_j(t)\in[z_j^-,z_j^+],j=1,2,\ldots,l\big\},$$
 and the terminal objective has the form
$$\phi(z)=\varphi(w(\mathsf{T})),$$
where $\varphi$ is convex and differentiable on the attainability set and $w(t)$ is a trajectory in $\mathfrak{L}^2([0,\mathsf{T}])$.\\

\noindent
Assume that this trajectory satisfies the following constraints
$$\dot{w}(t)=\mathcal{P}(t)w(t)+\mathcal{Q}(t)z(t),\quad w(0)=w_0,\quad t\in[0,\mathsf{T}],$$
where $\mathcal{P}(t)\in\mathcal{R}^{d\times d}$ and $\mathcal{Q}(t)\in\mathcal{R}^{d \times l}$ are continuous matrices for $t\in[0,\mathsf{T}]$. \\
It follows from the Pontryagin Maximum Principle that there exists
$v^* \in \mathfrak{L}^2([0,\mathsf{T}])$ such that for almost all
$t\in[0,\mathsf{T}]$, $(w^*,v^*,z^*)$ solves the system
\begin{equation}\label{state}
\qquad~
\begin{cases}
\begin{aligned}
    \dot{w}^*(t)&=\mathcal{P}(t)w^*(t)+\mathcal{Q}(t)z^*(t)\\
          w^*(0)&=x_0,\\
\end{aligned}
\end{cases}
\end{equation}
\begin{equation}\label{costate}
\begin{cases}
\begin{aligned}
    \dot{v}^*(t)&=-\mathcal{P}(t)' v^*(t)\\
          v^*(\mathsf{T})&=\nabla \varphi(w(\mathsf{T})),\\
\end{aligned}
\end{cases}\quad~
\end{equation}
\begin{equation}\label{vip0}
    0\in \mathcal{Q}(t)' v^*(t)+\mathcal{N}_{\widetilde{Z}}(z^*(t)),
\end{equation}
where $\mathcal{Q}(t)'$ means the transpose of $\mathcal{Q}(t)$ and $\mathcal{N}_{\widetilde{Z}}(z)$ is the normal cone to $\widetilde{Z}$ at $z$, which is maximal monotone.
Let $Tz(t):=\mathcal{Q}(t)' v(t)$. Then $T$ is the gradient of  $\phi$ (see \cite{VuongShehu} and the references therein). Hence \eqref{vip0} reduces to the monotone inclusion problem \eqref{MIP}, where $S=\mathcal{N}_{\widetilde{Z}}$ and $T=\nabla \phi$.

\noindent For our computational experiments, we discretize the continuous functions and choose a natural number $K$ with the \emph{mesh size}
$h:=\mathsf{T}/K$. We identify any discretized control $z^K:=(z_0,z_1,\ldots,z_{K-1})$ with its piecewise constant extension
$$z^K(t)=z_j~\text{for}~t\in\left[t_j,t_{j+1}\right),~j=0,1,\ldots K.$$
\noindent Also, we  identify any discretized state
$w^K:=(w_0,w_1,\ldots,w_{K})$  with its piecewise linear interpolation
$$w^K(t)=w_j+\frac{t-t_j}{h}\left(w_{j+1}-w_j\right),~\text{for}~t\in\left[t_j,t_{j+1}\right),~j=0,1,\ldots,K-1$$
\noindent  and  identify any discretized co-state variable $v^K:=(v_0,v_1,\ldots,v_{K})$ in a similar manner.\\
Then we adopt the Euler method for the  discretization (see \cite{OC2,OC1,VuongShehu} for more details).

\noindent Now, consider the following example where the terminal function is
nonlinear \cite[Example~6.3]{BressanPiccoli}:
\begin{equation}\label{P5.7}
    \begin{array}{ll}
        \mbox{minimize }    & -w_1(2) +\left(w_2(2)\right) ^2 \\
        \mbox{subject to }  & \dot{w_1}(t)=w_2(t), \\
                            & \dot{w_2}(t)=z(t),~\forall t\in[0,2], \\
                            & w_1(0)=0,~ w_2(0)=0,\\
                            & z(t) \in [-1,1].
    \end{array}
\end{equation}

\noindent The exact optimal control for Problem \eqref{P5.7} is
$$z^*=\begin{cases}
\begin{aligned}
    &1\ && \mbox{if }\, t\in [0, 1.2)\\
   -&1  && \mbox{if }\, t\in (1.2,2].
\end{aligned}
\end{cases}$$
\\ For this example, we take $K=100$ and randomly choose the initial controls  in $[-1, 1]$. We also take  $\delta_0 = 0.1, \delta_1 = 0.3, \bar{r}=0.3,  \sigma_n=\frac{0.005}{3n+25000}, c_n = \frac{1}{(n^2+1)}, \bar{\vartheta} = 0.04$ for Algorithms \ref{AL1} and \ref{AL2} while we take $\lambda_n = \frac{n+1}{15n+10}$ and $\gamma = 0.075$ for FRBSM \cite[Algorithm (2.2)]{MT} and RFBSM \cite[Algorithm (1.6)]{MT3}, respectively. The stopping criterion for this experiment is $\text{Tol}_n < 10^{-4}$, where $\text{Tol}_n=0.5\|w_n - J^S(w_n - Tw_n)\|^2$. Note that ${\text Tol}_n=0$ implies that $w_n\in (T+S)^{-1}(0)$. The numerical results  for the experiment are shown in Figure \ref{Fig2**}.

\begin{figure}
			\begin{center}
				\includegraphics[height=5cm]{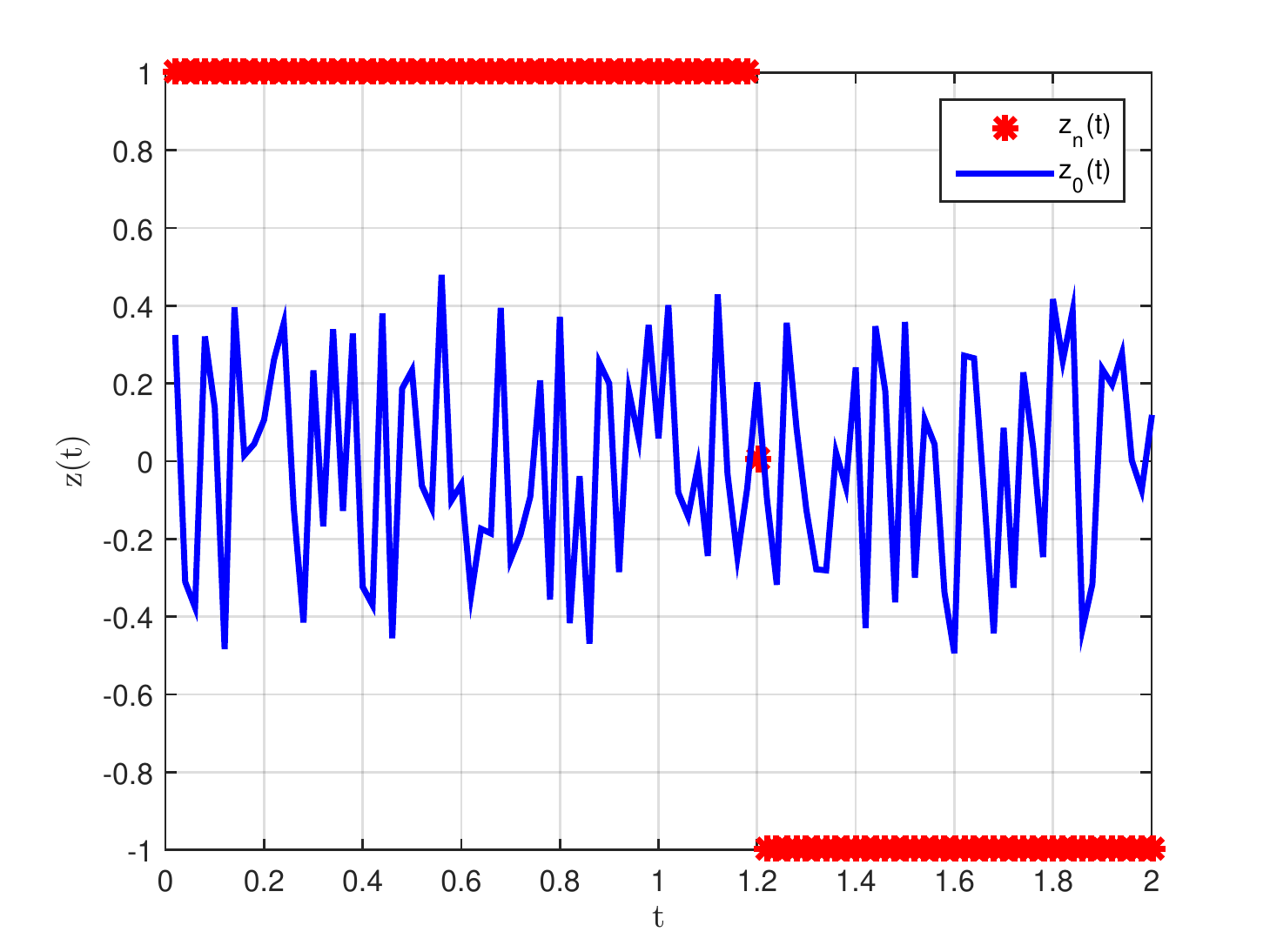}	
				\includegraphics[height=5cm]{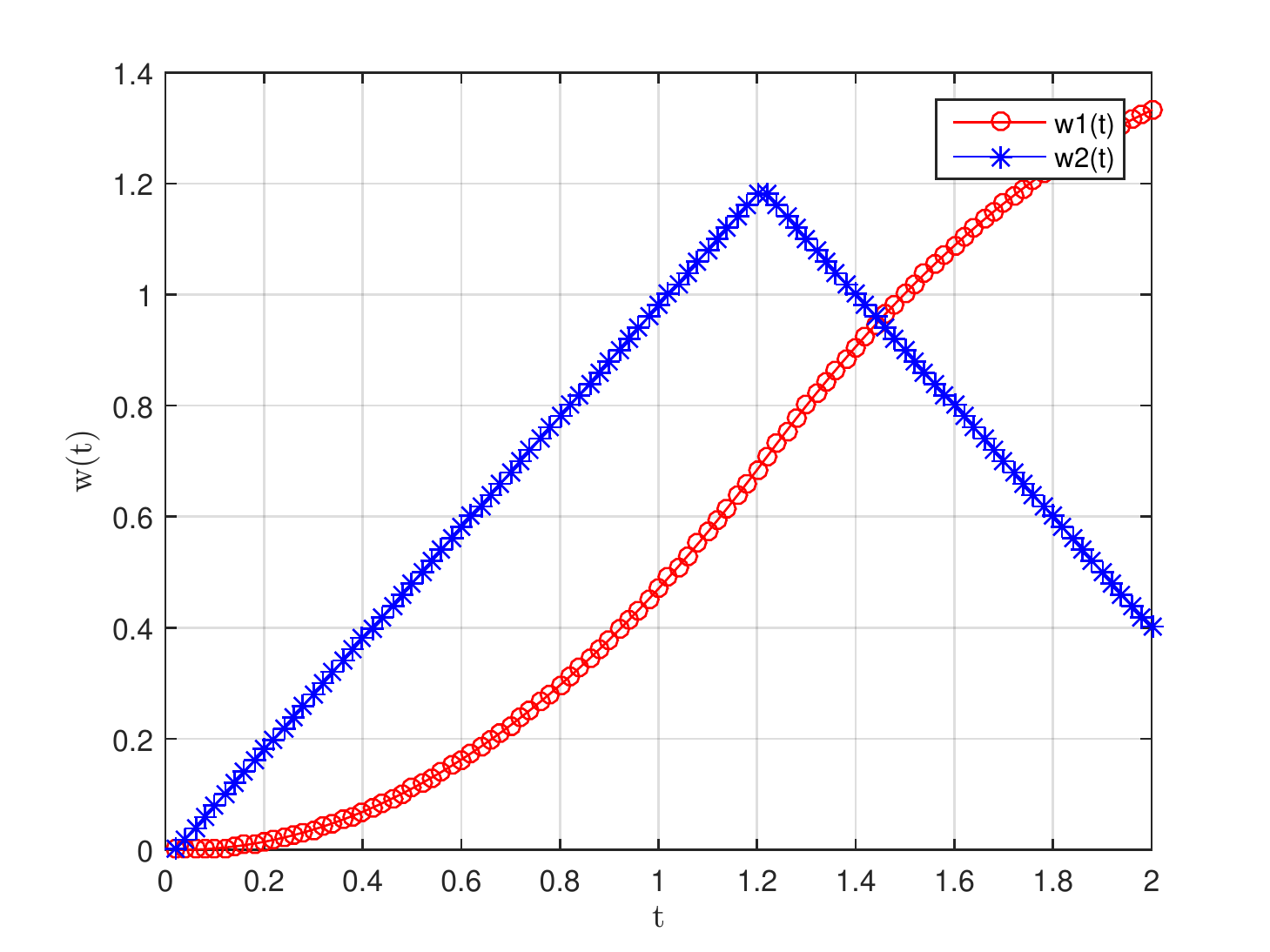}\\
							\includegraphics[height=5cm]{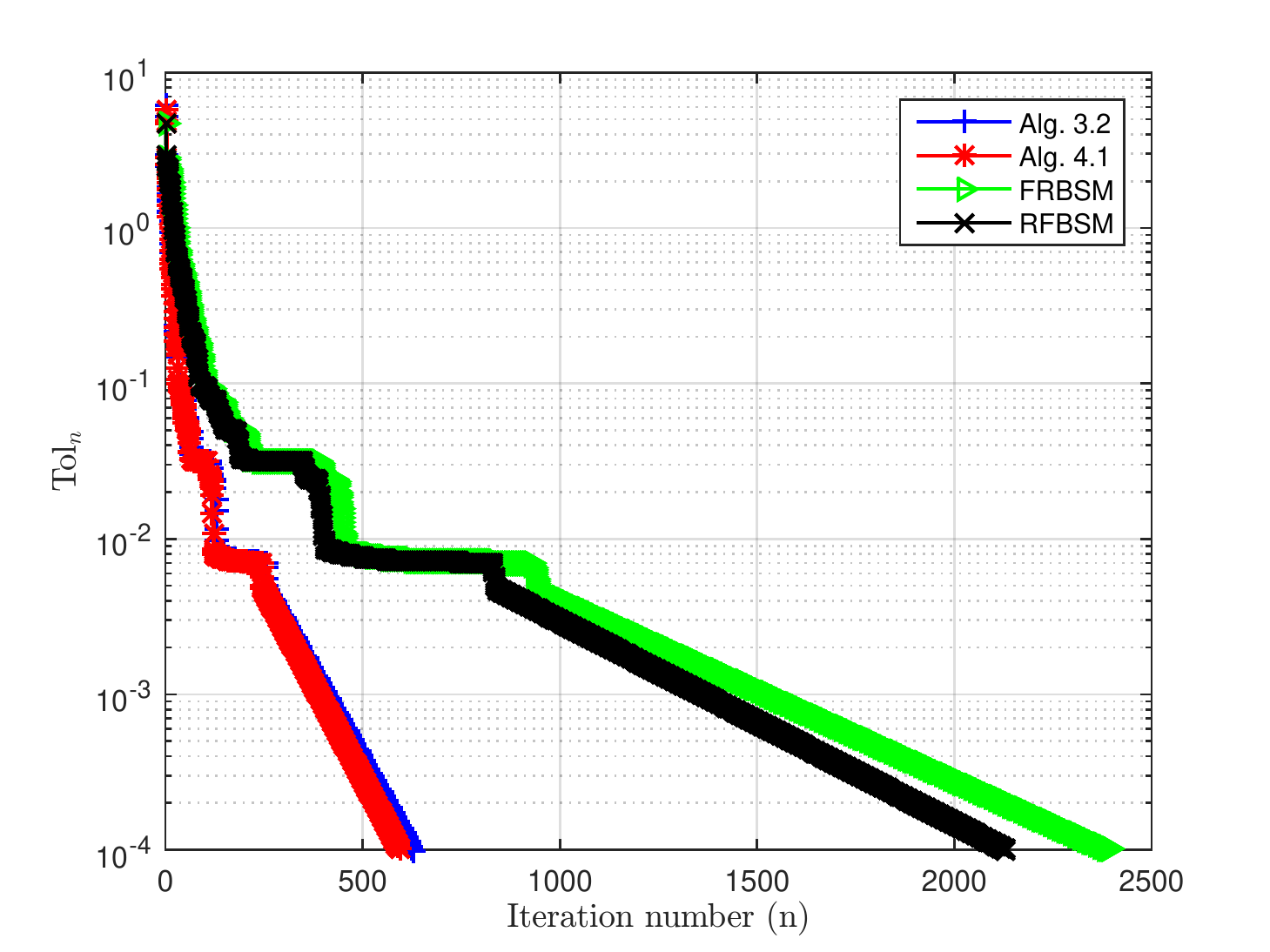}
			\end{center}
			\caption{Numerical results for {\bf Example \ref{OCP}:} Top Left: random initial control (blue) and optimal control (red); Top Right:  optimal trajectories (Top Left and Top Right computed by Algorithm \ref{AL1});
			Bottom: comparison of Algorithm \ref{AL1}, Algorithm \ref{AL2}, FRBSM and RFBSM with CPU time (sec), 1.5656, 1.3799, 4.1210 and 3.6200, respectively.}\label{Fig2**}
\end{figure}

\hfill

\end{example}

\subsection{Academic Examples}
\begin{example}\label{AE1}
Consider the following convex minimization problem:
\begin{equation}\label{CMP}
\min_{w\in \mathcal{R}^3}\|w\|^2_2 + (-7,1,-3)w + 4w_1 +\|w\|_1,
\end{equation}
where $w = (w_1, w_2, w_3)\in \mathcal{R}^3$.

\noindent
Set $\mathcal{P}w = \|w\|^2_2 + (-7,1,-3)w + 4w_1$ and $\mathcal{D}w = \|w\|_1$. Then $\mathcal{P}$ is convex and continuously differentiable on $\mathcal{R}^3$ with $\nabla \mathcal{P}w = 2w + (-3,1,-3)$, and $\mathcal{D}$ is subdifferentiable.\\
 Note that $\nabla \mathcal{P}$ is monotone and $2$-Lipschitz continuous, and $\partial \mathcal{D}$ is maximal monotone, where $\nabla \mathcal{P}$ and $\partial \mathcal{D}$ are the gradient and subdifferential of $\mathcal{P}$ and $\mathcal{D}$, respectively.\\
  Note also that this problem is equivalent to the following inclusion problem:
\begin{equation*}
0\in \left(\nabla \mathcal{P} + \partial \mathcal{D}\right)w.
\end{equation*}
It is known that
\begin{equation*}
J_{\bar{\delta}}^{S}(w)=(I_{\mathcal{R}^3}+ \bar{\delta}\partial \mathcal{D})^{-1}(w) = (\mu_1, \mu_2, \mu_3),
\end{equation*}
where
$$\mu_i = \text{sign}(w_i)\max\{|w_i|-\bar{\delta}, 0\}, i = 1, 2, 3.$$
Thus, setting $\nabla \mathcal{P} = T$ and $\partial \mathcal{D} = S$, we can employ our methods to solve the above convex minimization problem.\\
For the experiment of this example, we choose the same parameters as in Example \ref{OCP}. Furthermore, we choose  $\hat{v} = (2, 1, -6)$ and the initial values as follows:\\
\noindent {\it Case Ia:} $w_0 = (5, 1, -4), w_1 = (20, 4, 7)$;\\
\noindent {\it Case Ib:} $w_0 = (-7, 10, 3), w_1 = (90, -3, -4)$.

\noindent
The values of the $n$-th iterates $w_n$ for each choice of initial values is given in Figure \ref{Fig3a} and Table \ref{Tab3a} up to $100$ iterations with $D_n=0.5\|w_n - J^S(w_n - Tw_n)\|^2$.\\
Actually, the minimizer of the convex minimization problem \eqref{CMP} is $(1,0,1)$ (see Table \ref{Tab3a}). Thus, we set $\text{Tol}_n = 0.5\|w_n - (1,0,1)\|^2 $ and use the stopping criterion $\text{Tol}_n< 10^{-10}$. Furthermore, we plot the graph of $\text{Tol}_n$ against number of iterations in Figure \ref{Fig3b} with corresponding numerical reports in Table \ref{Tab3b}.

\begin{figure}
			\begin{center}
				\includegraphics[height=5cm]{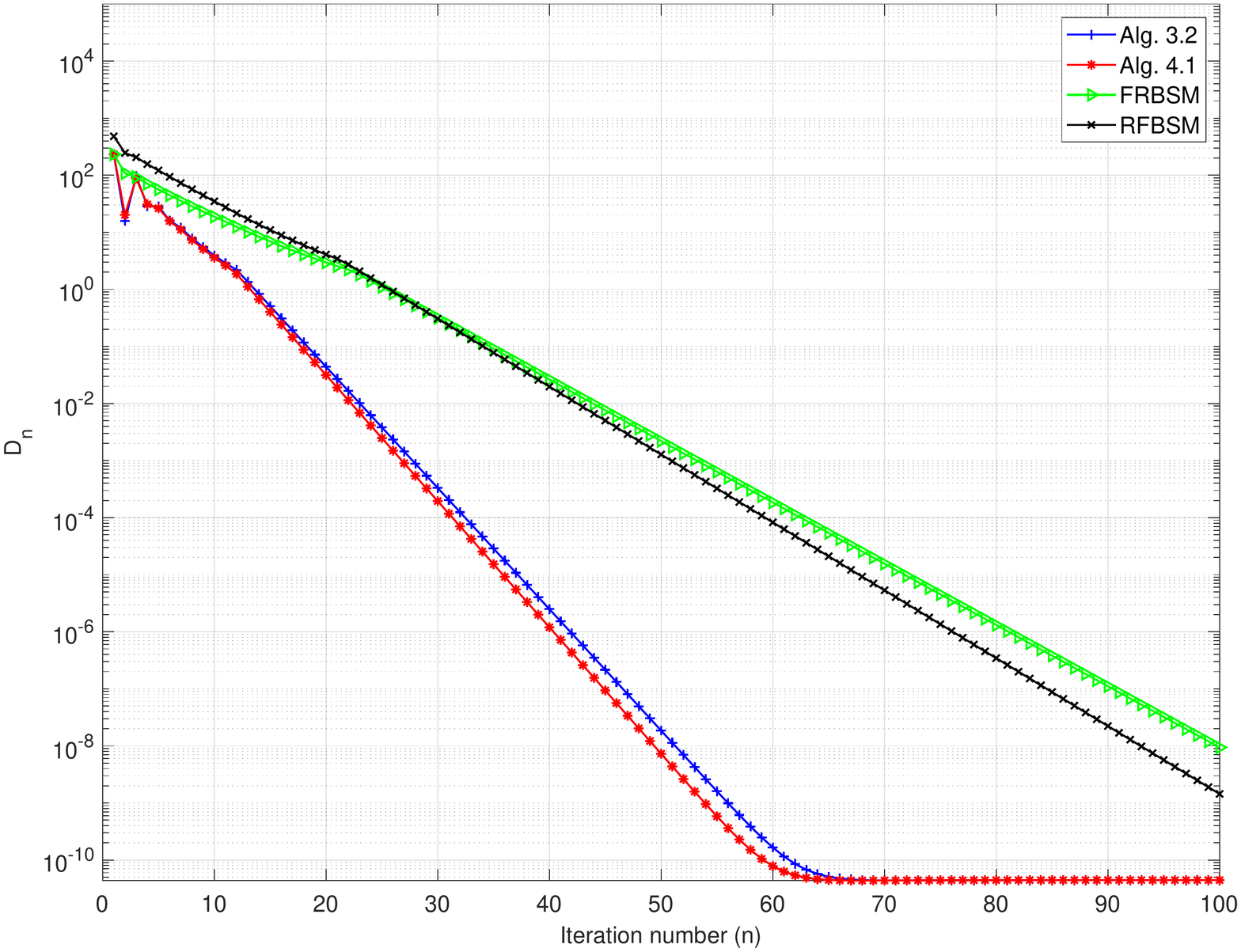}	
				\includegraphics[height=5cm]{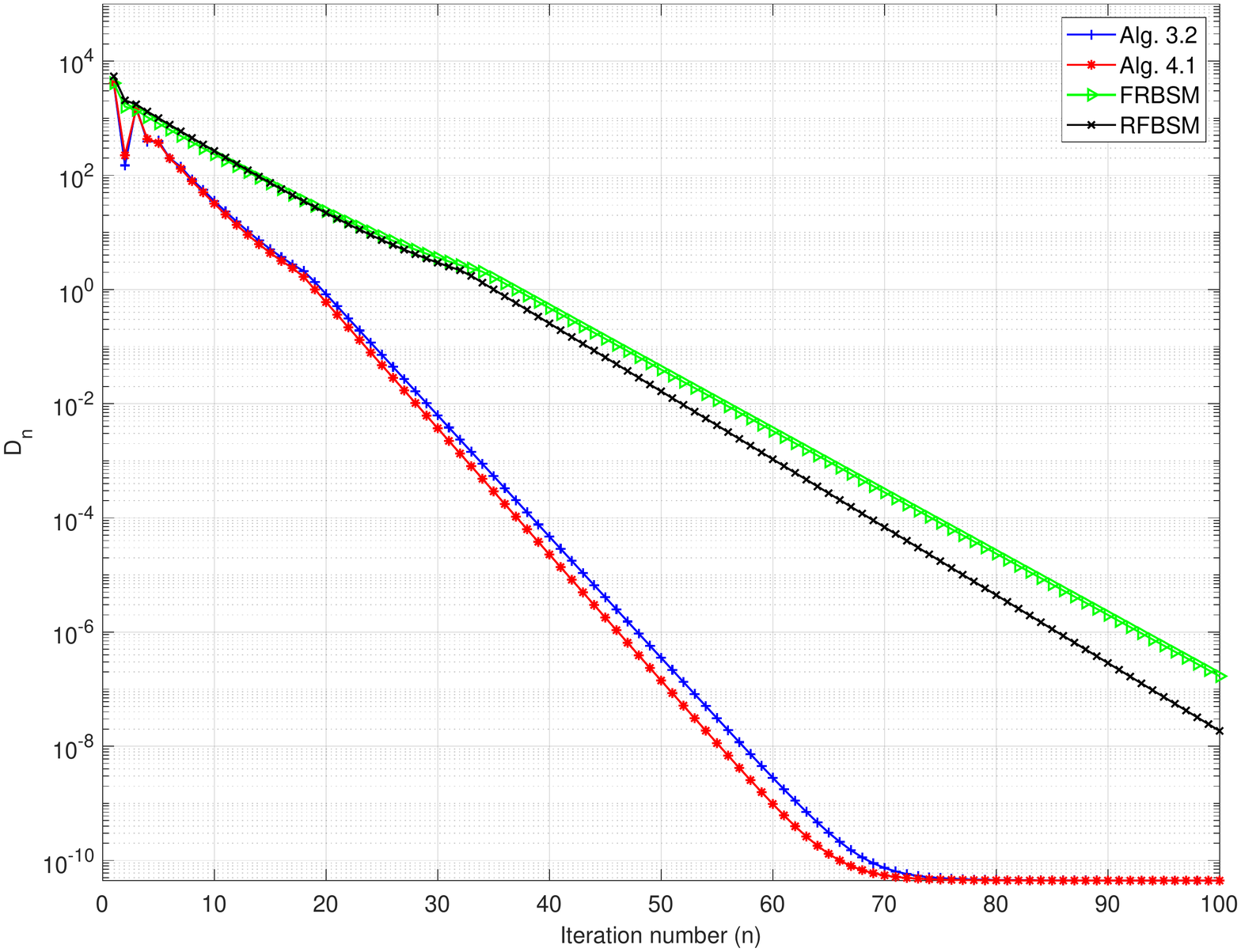}
			\end{center}
			\caption{The behavior of $D_n$ for {\bf Example \ref{AE1}:} Top Left: Case Ia; Top Right: Case Ib.}\label{Fig3a}
		\end{figure}

\begin{table}[h]
			\caption{Numerical results for {\bf Example \ref{AE1}:} Showing different values of the $n$-th iterates $w_n$.}
			\label{Tab3a}
			\begin{tabular}{ |p{1cm} |p{0.8cm}| p{3cm}| p{3cm}| p{3.6cm}|p{3.6cm}|}
			\hline
			\noindent & $n$ & Alg. \ref{AL1} &   Alg. \ref{AL2}& FRBSM & RFBSM \\
			\hline
			Ia  &   10 & (2.2949, 0, 1.3516) & (2.2850, 0, 1.3499) & (5.1660, 0.1442, 1.9988) & (5.1808, 0.1348, 2.0915) \\
			& 20& (1.1122, 0, 1.0304) & (1.1103, 0, 1.0300) & (2.1861, 0, 1.1284) & (2.0629, 0, 1.2775)\\
			& 50 & (1.0001, 0, 1.0000) & (1.0001, 0, 1.0000) & (1.0285, 0, 1.0068) & (1.0175, 0, 1.0046)\\
			& 70 & (1.0000, 0, 1.0000) & (1.0000, 0, 1.0000) & (1.0024, 0, 1.0006) & (1.0011, 0, 1.0003)\\
			& 100 & (1.0000, 0, 1.0000) & (1.0000, 0, 1.0000) & (1.0001, 0, 1.0000) & (1.0000, 0, 1.0000)\\
			\hline
			Ib  &   10 & (6.8208, 0, 0.8831) & (6.7805, 0, 0.8829) & (19.1621, 0, 0.6536) & (19.6069, 0, 0.6727) \\
			& 20 & (1.5042, 0, 0.9899) & (1.4961, 0, 0.9899) & (6.1708, 0, 0.9014) & (5.7306, 0, 0.9168)\\
			& 50 & (1.0003, 0, 1.0000) & (1.0003, 0, 1.0000) & (1.1243, 0, 0.9976) & (1.0777, 0, 0.9986)\\
			& 70 & (1.0000, 0, 1.0000) & (1.0000, 0, 1.0000) & (1.0105, 0, 0.9998) & (1.0050, 0, 0.9999)\\
			& 100 & (1.0000, 0, 1.0000) & (1.0000, 0, 1.0000) & (1.0003, 0, 1.0000) & (1.0000, 0, 1.0000)\\
			\hline
			\end{tabular}
		\end{table}

\begin{figure}
			\begin{center}
				\includegraphics[height=5cm]{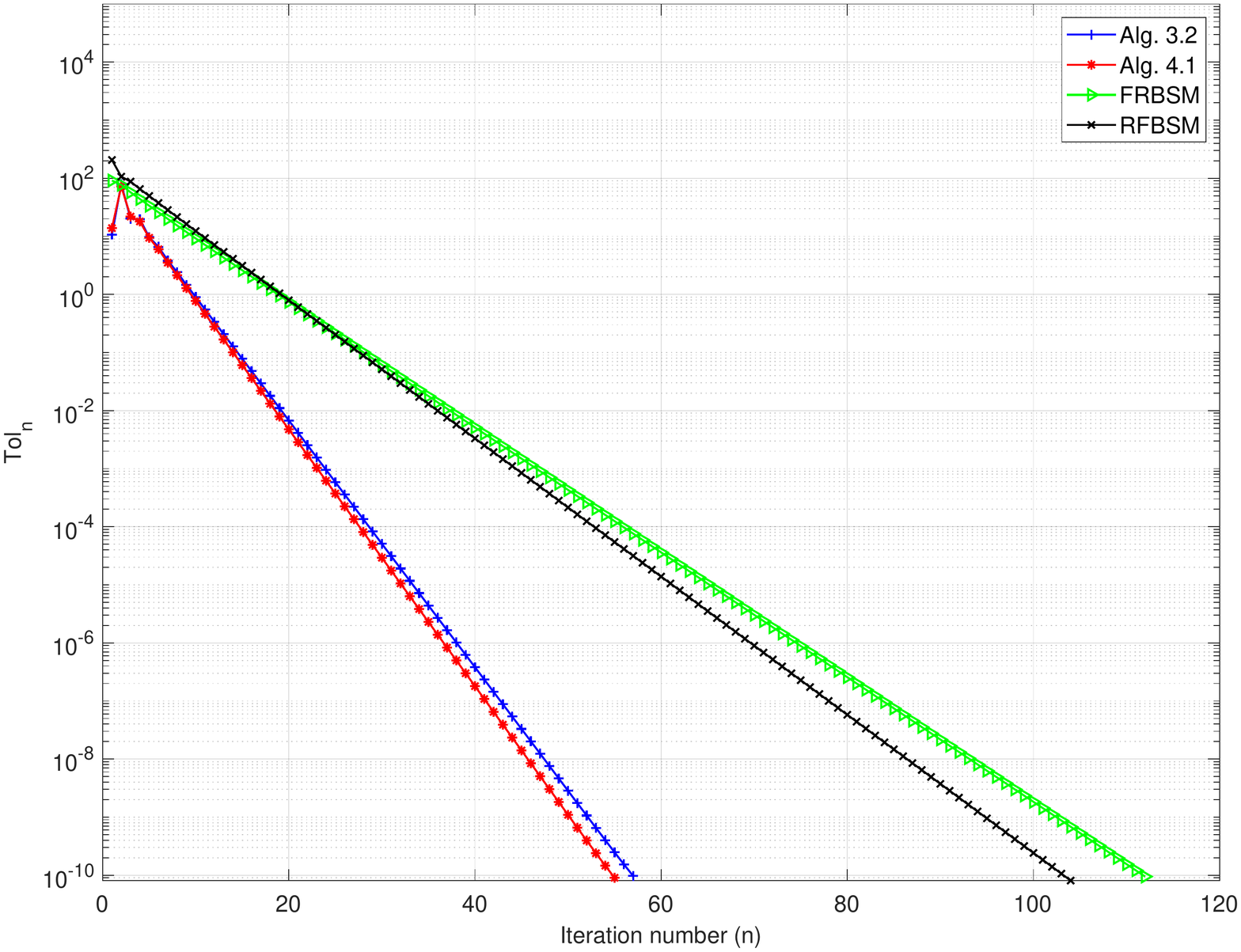}	
				\includegraphics[height=5cm]{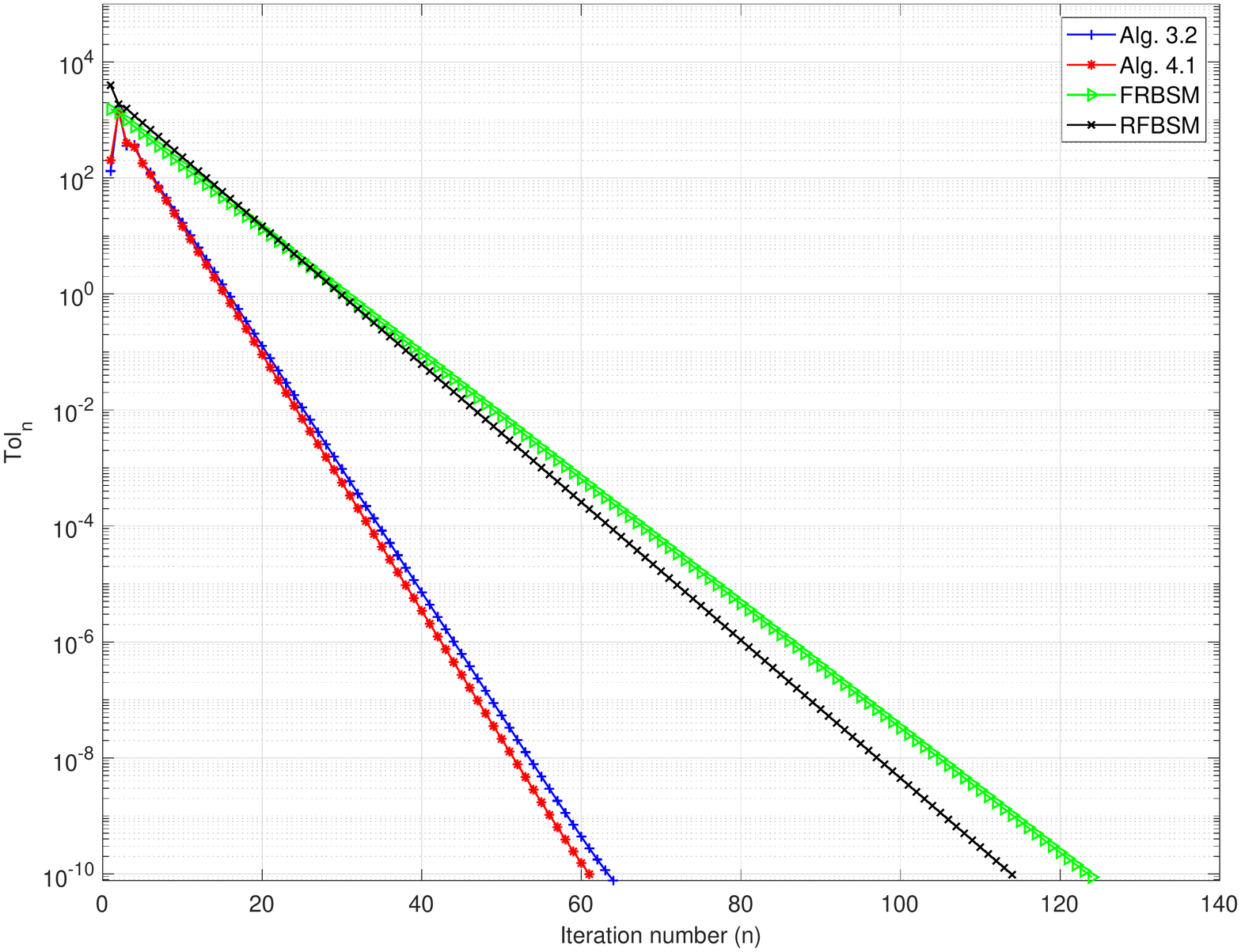}
			\end{center}
			\caption{The behavior of $\text{Tol}_n$ for {\bf Example \ref{AE1}:} Top Left: Case Ia; Top Right: Case Ib.}\label{Fig3b}
\end{figure}

\hfill

\begin{table}[h]
			\caption{Numerical results for {\bf Example \ref{AE1}} with $\text{Tol}_n< 10^{-10}$.}
			\label{Tab3b}
			\begin{tabular}{ |p{1.5cm} |p{2.5cm}| p{1.6cm}| p{1.8cm}| p{1.6cm}|p{1.6cm}|}
				\hline
				\noindent & \noindent & Alg. \ref{AL1} &   Alg. \ref{AL2} & FRBSM & RFBSM \\
				\hline
				Case Ia  &   CPU time (sec)  & 0.0024 & 0.0018 & 0.0024 & 0.0036 \\
				& No of Iter.  & 57 & 55 & 112 & 104 \\
				\hline
				Case Ib &   CPU time (sec)  & 0.0033 & 0.0024 & 0.0033 & 0.0790\\
				& No. of Iter. & 64 & 61 & 124 & 114\\
				\hline
	\end{tabular}
\end{table}
\end{example}

\hfill

 \begin{example}\label{EX2}
 Let $\mathcal{H}=\left(\mathfrak{l}_2(\mathbb{R}), ~\|. \|_{\mathfrak{l}_2}\right)$,
  where  $\mathfrak{l}_2(\mathbb{R}):=\{a=(a_1, a_2,...,a_i,...),~~a_i \in \mathbb{R} :\sum\limits_{i=1}^{\infty}|a_i|^2<\infty\}$   and   $||a||_{\mathfrak{l}_2} := \left( \sum\limits_{i=1}^\infty |a_i|^2 \right)^{\frac{1}{2}} \; ~~\forall a \in \mathfrak{l}_2(\mathbb{R}).$\\
    Let $S,T:\mathfrak{l}_2\to \mathfrak{l}_2$ be defined by
  $$Sa := \left(2a_1, 2a_2,...,2a_i,...\right), ~\forall a\in \mathfrak{l}_2$$
and
$$Ta := \left(\frac{a_1+|a_1|}{2}, \frac{a_2+|a_2|}{2},...,\frac{a_i+|a_i|}{2},...\right), ~\forall  a\in \mathfrak{l}_2.$$
Then $S$ is maximal monotone and $T$ is Lipschitz continuous and monotone with Lipschitz constant $\mathsf{L}=1$.

\noindent  We choose the following initial values:\\
 \noindent {\it Case IIa:} $w_0 = (2, -1, \frac{1}{2}, -\frac{1}{4}, \dots), w_1 = (\frac{2}{3}, \frac{1}{9}, \frac{1}{54}, \frac{1}{324},\dots)$;\\
 \noindent {\it Case IIb:} $w_0 = (4, 1, \frac{1}{4}, \frac{1}{16}, \dots), w_1 = (9, 3\sqrt{3}, 3, \sqrt{3}, \dots)$;\\
 \noindent {\it Case IIc:} $w_0 = (\frac{4}{3}, \frac{4}{9}, \frac{4}{27}, \frac{4}{81}, \dots), w_1 = (-2, 1, -\frac{1}{2}, \frac{1}{4}, \dots)$;\\
\noindent {\it Case IId:} $w_0 = (-4, 1, -\frac{1}{4}, \frac{1}{16}, \dots), w_1 = (20, -4, \frac{4}{5}, -\frac{4}{25}, \dots)$.

\noindent
Also we choose $\delta_0 = \frac{1}{101}, \delta_1 = \frac{2}{201}, \bar{r}=0.15, \hat{v} = (\frac{3}{2}, -\frac{3}{4}, \frac{3}{8}, -\frac{3}{16},\dots), \sigma_n=\frac{0.005}{3n+25000},  c_n = \frac{1}{(10n+77)^2}, \bar{\vartheta} = 0.04$ for Algorithms \ref{AL1} and \ref{AL2} while we  take $\lambda_n = \frac{n+1}{100n+101}$ and $\gamma = \frac{2}{201}$ for FRBSM \cite[Algorithm (2.2)]{MT} and RFBSM \cite[Algorithm (1.6)]{MT3}, respectively. The stopping criterion for this example is $\text{Tol}_n < 10^{-8}$, where $\text{Tol}_n=0.5\|w_n - J^S(w_n - Tw_n)\|^2$.  The numerical results  are shown in Figure \ref{Fig2} and Table \ref{Tab2}.
\begin{figure}
			\begin{center}
				\includegraphics[height=5cm]{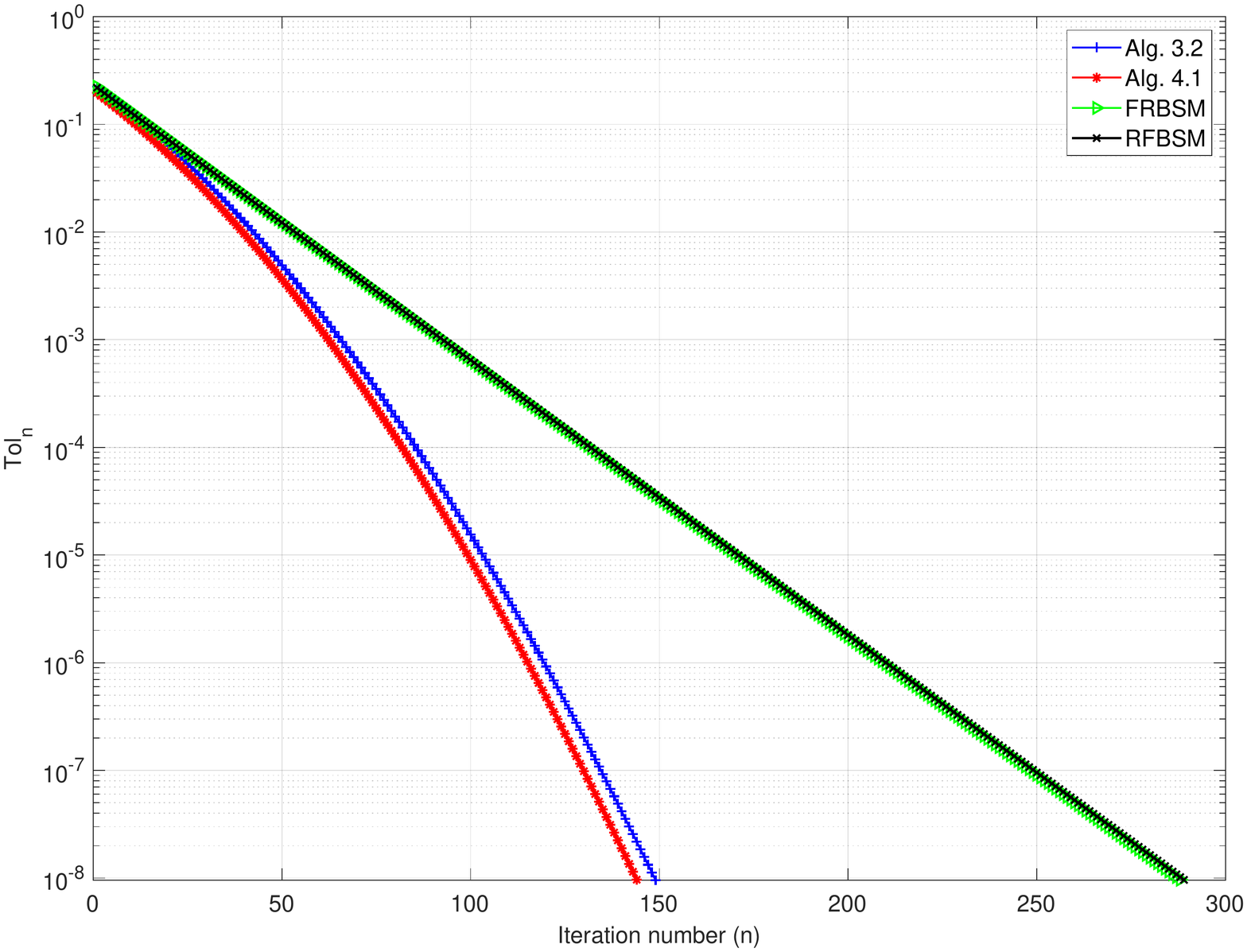}	
				\includegraphics[height=5cm]{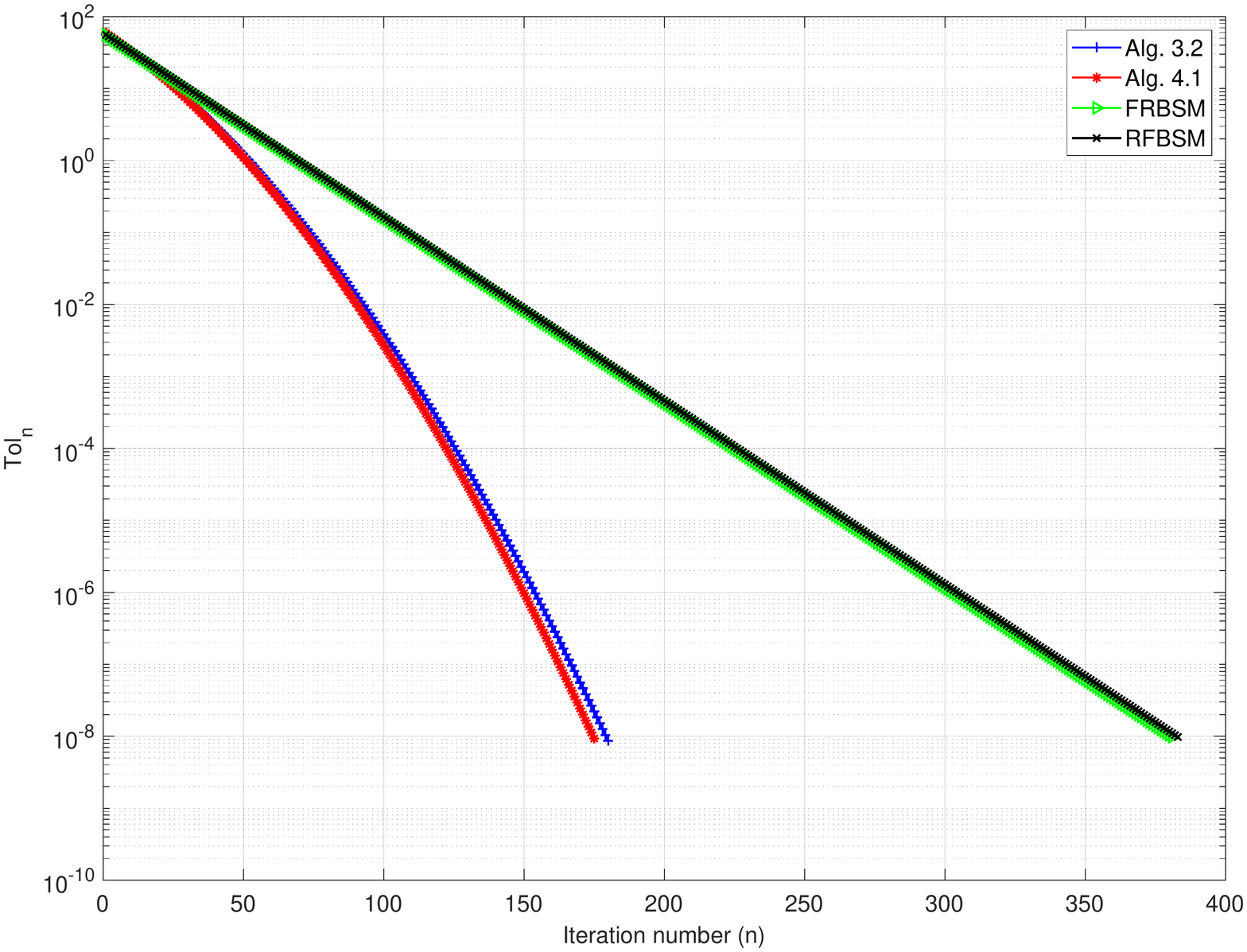}\\
				\includegraphics[height=5cm]{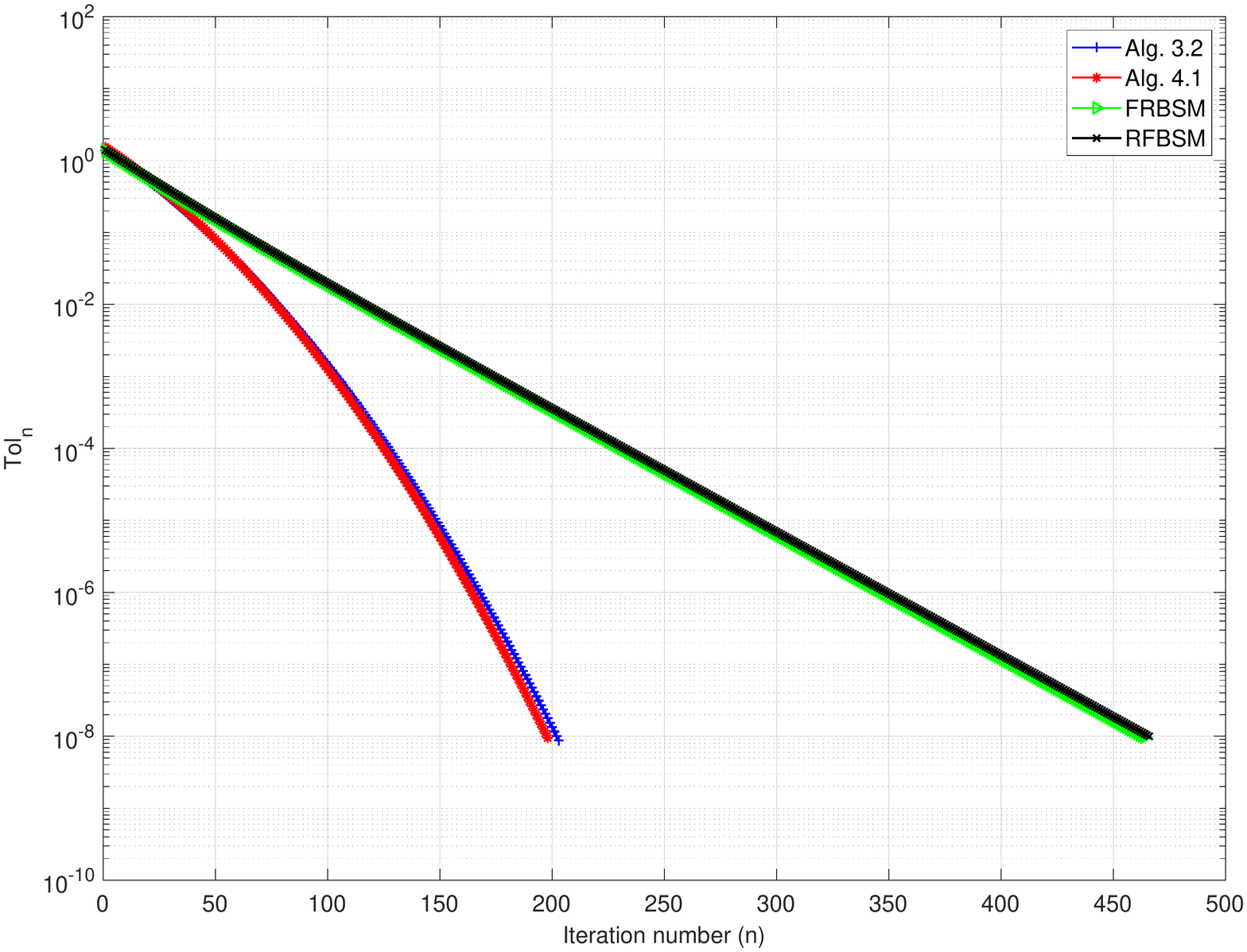}	
				\includegraphics[height=5cm]{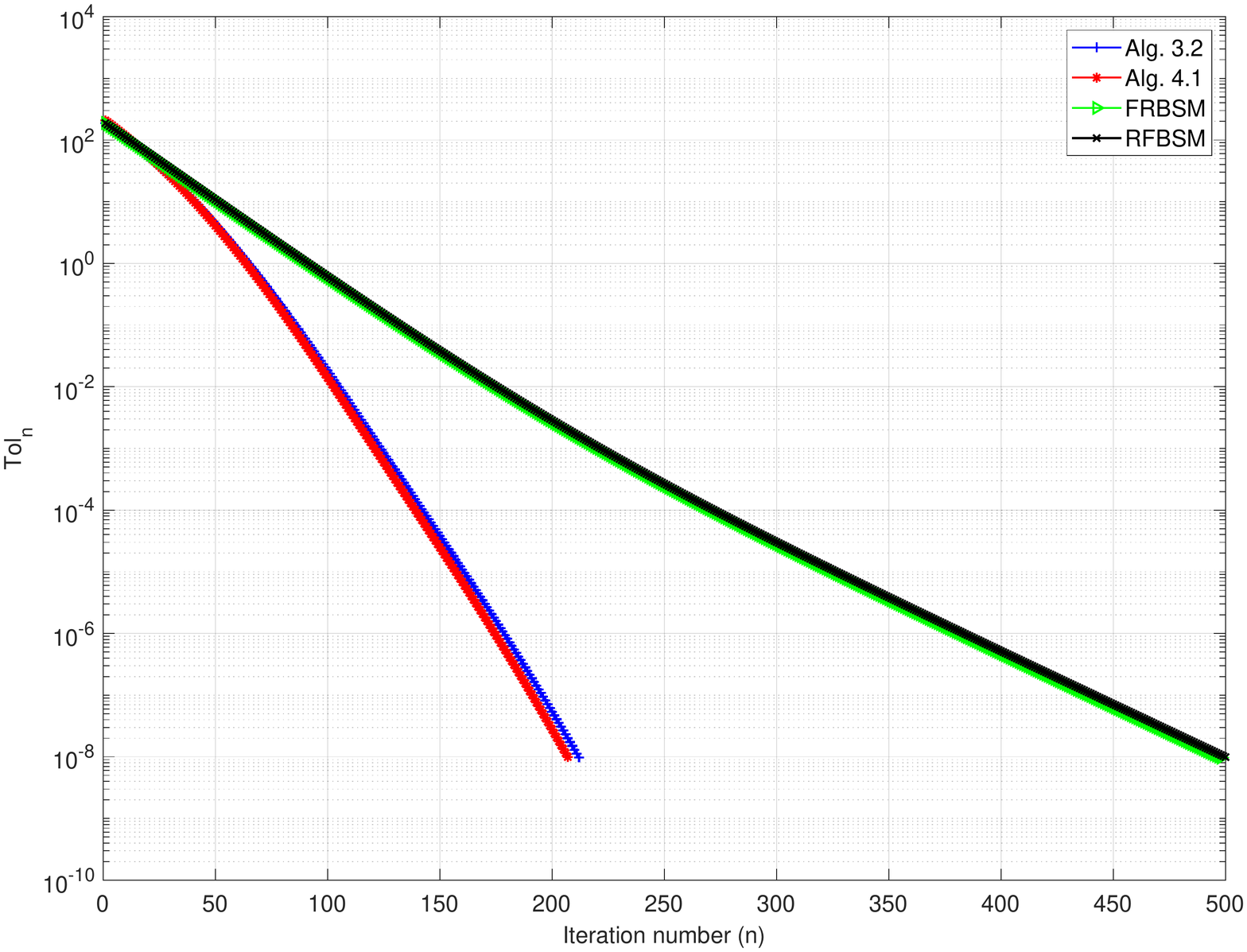}
			\end{center}
			\caption{The behavior of $\text{Tol}_n$ for {\bf Example \ref{EX2}:} Top Left: Case IIa; Top Right: Case IIb; Bottom Left:  Case IIc; Bottom Right: Case IId.}\label{Fig2}
\end{figure}

\begin{table}[h]
			\caption{Numerical results for {\bf Example \ref{EX2}} with $\text{Tol}_n< 10^{-8}$.}
			\label{Tab2}
			\begin{tabular}{ |p{1.5cm} |p{2.5cm}| p{1.6cm}| p{1.8cm}| p{1.6cm}|p{1.6cm}|}
				\hline
				\noindent & \noindent & Alg. \ref{AL1} &   Alg. \ref{AL2}& FRBSM & RFBSM \\
				\hline
				Case IIa  &   CPU time (sec)  & 0.0029 & 0.0027 & 0.0032 & 0.0035 \\
				& No of Iter.  & 149 & 149 & 288 & 289\\
				\hline
				Case IIb &   CPU time (sec)  & 0.0039 & 0.0034 & 0.0048 & 0.0041 \\
				& No. of Iter. & 180 & 175 & 381 & 383\\
				\hline
				Case IIc & CPU time (sec)& 0.0062 & 0.0059 & 0.0071 & 0.0076 \\
				& No of Iter. & 203 & 198 & 464 & 466\\
				\hline
				Case IId & CPU time (sec)  & 0.0021 & 0.0019 & 0.0025 & 0.0586\\
				& No of Iter. & 212 & 207 & 498 & 500\\
				\hline
			\end{tabular}
\end{table}
\end{example}

\section{Conclusion and future research} \label{Se6}
\noindent
We have proposed several new methods for solving the monotone inclusion problem \eqref{MIP} in a real Hilbert space. The first method{\color{blue},} which we called
a {\it forward-reflected-anchored-backward splitting method}{\color{blue},} inherits the attractive features of the forward-reflected-backward splitting method \eqref{FRB}, namely, it only involves one forward evaluation of the single-valued operator and one backward evaluation of the set-valued operator, and does not require the cocoercivity of the single-valued operator, but still converges strongly rather than weakly. The other methods of this paper are the inertial, viscosity and inertial viscosity variants of the first one. These variants share the same attractive features of the first method, and they also converge strongly.\\
Part of our future research is to study the rate of convergence of the proposed methods of this paper.  \\
It would be of interest to incorporate perturbations and error terms to these  methods because computing the resolvent of the set-valued operator may be difficult in some applications.\\
Finally, it would also be of interest to develop anchored (Halpern-type) and viscosity-type variants of the {\it \underline{G}olden \underline{RA}tio \underline{AL}gorithm} (GRAAL \cite{GRAAL}) for solving the monotone inclusion problem \eqref{MIP} and establish their strong convergence.

\hfill

\noindent
\textbf{Declarations}
 \vskip 0.2in

\noindent \textbf{Funding}: The second author was partially supported by the Israel Science Foundation (Grant 820/17), by the Fund for the Promotion of Research at the Technion and by the Technion General Research Fund.\\

\noindent
\textbf{Availability of data and material}: Not applicable.\\

\noindent
\textbf{Code availability}: The Matlab codes employed to run the numerical experiments are available upon request to
the authors.\\

\noindent
\textbf{Conflict of interest}: The authors declare that they have no known competing financial interests or personal relationships that could have appeared to influence the work reported in this paper.

	\end{document}